\documentclass[12pt, reqno]{amsart}
\usepackage{amssymb, mathrsfs}
\usepackage{latexsym}
\usepackage{amsmath}
\usepackage{amsthm}
\usepackage{amsfonts}
\usepackage{dsfont, upgreek}
\usepackage{mathtools}
\usepackage{epsfig}
\usepackage{amscd}
\usepackage{graphicx}
\usepackage{tikz-cd}
\usepackage{enumitem}
\setlist[enumerate]{itemsep=0.5ex}

\usepackage{color}
\usepackage{tikz}
\usetikzlibrary{patterns}

\usepackage[left=1.4in,right=1.4in,top=1.2in,bottom=1.2in]{geometry}

\usepackage{tikz}
\usetikzlibrary{decorations.markings}
\usetikzlibrary{arrows.meta}

\usepackage{extarrows}

\usepackage{adjustbox}
\usepackage{pgfplots}
\usepgfplotslibrary{colormaps}
\usepackage{slashed}

\usepackage[colorlinks=true,linkcolor=blue,citecolor=blue]{hyperref}

\usepackage[colorinlistoftodos,prependcaption,textsize=tiny]{todonotes}  

\usepackage[all,cmtip]{xy}
\theoremstyle{plain}
\newtheorem{theorem}{Theorem}[section]
\newtheorem{proposition}[theorem]{Proposition}
\newtheorem{lemma}[theorem]{Lemma}

\newtheorem{conjecture}[theorem]{Conjecture}

\theoremstyle{definition} 
\newtheorem{definition}[theorem]{Definition}

\newtheorem*{claim*}{Claim}

\theoremstyle{remark} 
\newtheorem{remark}[theorem]{Remark}

\numberwithin{equation}{section}

\newcommand{\Sc}{\mathrm{Sc}}

\newcommand{\Bigwedge}{\mathord{\adjustbox{raise=.4ex, totalheight=.7\baselineskip}{$\bigwedge$}}}

\newcommand{\fiber}{\mathbb F}

\newcommand{\grad}{\mathrm{grad}}

\newcommand{\id}{\mathrm{id}}

\newcommand{\R}{\mathbb{R}}
\newcommand{\tr}{\mathrm{tr}}

\newcommand{\ncon}{\prescript{N}{}{\nabla}}

\newcommand{\interior}[1]{%
	{\kern0pt#1}^{\mathrm{\,o}}%
}

\makeatletter
\let\save@mathaccent\mathaccent
\newcommand*\if@single[3]{%
	\setbox0\hbox{${\mathaccent"0362{#1}}^H$}%
	\setbox2\hbox{${\mathaccent"0362{\kern0pt#1}}^H$}%
	\ifdim\ht0=\ht2 #3\else #2\fi
}
\newcommand*\rel@kern[1]{\kern#1\dimexpr\macc@kerna}
\newcommand*\overbar[1]{\@ifnextchar^{{\wide@bar{#1}{0}}}{\wide@bar{#1}{1}}}
\newcommand*\wide@bar[2]{\if@single{#1}{\wide@bar@{#1}{#2}{1}}{\wide@bar@{#1}{#2}{2}}}
\newcommand*\wide@bar@[3]{%
	\begingroup
	\def\mathaccent##1##2{%
		\let\mathaccent\save@mathaccent
		\if#32 \let\macc@nucleus\first@char \fi
		\setbox\z@\hbox{$\macc@style{\macc@nucleus}_{}$}%
		\setbox\tw@\hbox{$\macc@style{\macc@nucleus}{}_{}$}%
		\dimen@\wd\tw@
		\advance\dimen@-\wd\z@
		\divide\dimen@ 3
		\@tempdima\wd\tw@
		\advance\@tempdima-\scriptspace
		\divide\@tempdima 10
		\advance\dimen@-\@tempdima
		\ifdim\dimen@>\z@ \dimen@0pt\fi
		\rel@kern{0.6}\kern-\dimen@
		\if#31
		\overline{\rel@kern{-0.6}\kern\dimen@\macc@nucleus\rel@kern{0.4}\kern\dimen@}%
		\advance\dimen@0.4\dimexpr\macc@kerna
		\let\final@kern#2%
		\ifdim\dimen@<\z@ \let\final@kern1\fi
		\if\final@kern1 \kern-\dimen@\fi
		\else
		\overline{\rel@kern{-0.6}\kern\dimen@#1}%
		\fi
	}%
	\macc@depth\@ne
	\let\math@bgroup\@empty \let\math@egroup\macc@set@skewchar
	\mathsurround\z@ \frozen@everymath{\mathgroup\macc@group\relax}%
	\macc@set@skewchar\relax
	\let\mathaccentV\macc@nested@a
	\if#31
	\macc@nested@a\relax111{#1}%
	\else
	\def\gobble@till@marker##1\endmarker{}%
	\futurelet\first@char\gobble@till@marker#1\endmarker
	\ifcat\noexpand\first@char A\else
	\def\first@char{}%
	\fi
	\macc@nested@a\relax111{\first@char}%
	\fi
	\endgroup
}
\makeatother

\begin{document}

\title[Flat corner domination conjecture and Stoker conjecture]{On Gromov's flat corner domination conjecture  and Stoker's conjecture}

\author{Jinmin Wang}
\address[Jinmin Wang]{Department of Mathematics, Texas A\&M University}
\email{jinmin@tamu.edu}
\thanks{The first author is partially supported by NSF 1952693.}
\author{Zhizhang Xie}
\address[Zhizhang Xie]{ Department of Mathematics, Texas A\&M University }
\email{xie@math.tamu.edu}
\thanks{The second author is partially supported by NSF 1800737 and 1952693.}

\begin{abstract}
	In this paper, we prove Gromov's flat corner domination conjecture  in all dimensions. As a consequence, we  answer positively the Stoker conjecture for convex Euclidean polyhedra  in all dimensions. By applying the same techniques, we also prove a rigidity theorem for strictly convex domains in Euclidean spaces.
\end{abstract}
\maketitle
	\section{Introduction}

The main purpose of this paper is to  solve  Gromov's flat corner domination  conjecture (Conjecture \ref{conj:flatcorner})  in all dimensions (Theorem \ref{thm:polyhedra}). As a consequence, we answer positively the 	 Stoker conjecture for convex Euclidean polyhedra  in all dimensions (Theorem \ref{thm:stokereuclidean}). The Stoker conjecture is a question concerning the rigidity of convex Euclidean polyhedra, which roughly says that  the shape of a convex Euclidean polyhedron is determined by its dihedral angles.  See  Theorem \ref{thm:stokereuclidean} below for the precise statement. The conjecture has  attracted a lot of attention over  the past more-than-fifty years since it was proposed by Stoker in 1968 (see the brief discussion after Theorem \ref{thm:stokereuclidean}).  Despite of all the efforts, the  conjecture had resisted all previous attempts. Our approach to  the Stoker conjecture in the present paper is to view the conjecture as a special case in the larger context of comparison problems of scalar curvature, mean curvature and dihedral angles, where the latter is a program prompted by Gromov \cite{GromovDiracandPlateau, Gromovinequalities2018, 	Gromov4lectures2019} and has inspired a wave of research activity in recent years. By new  index theoretic methods, we answer positively Gromov's flat corner domination conjecture in all dimensions. As a consequence, we obtain a positive solution to  the Stoker conjecture  in all dimensions.

 Gromov's flat corner domination  conjecture (Conjecture \ref{conj:flatcorner} below) is one of the fundamental conjectures  among the extensive list of conjectures and open questions on  scalar curvature formulated by Gromov  \cite{GromovDiracandPlateau, Gromovinequalities2018, 	Gromov4lectures2019}. It is closely related to   Gromov's dihedral extremality and rigidity conjectures (see Conjecture \ref{conj:extremal} and Conjecture \ref{conj:dihedral} below). More precisely, Gromov's flat corner domination  conjecture is a  stronger version of  Gromov's dihedral rigidity conjecture. All three conjectures concern the comparisons of scalar curvature, mean curvature and dihedral angles for Riemannian metrics on polyhedra. They can be viewed as scalar curvature analogue of the Alexandrov’s triangle comparisons for spaces whose sectional curvature is bounded below \cite{MR0049584}. These conjectures of Gromov have  profound implications in  geometry and mathematical physics. For example, it implies the positive mass theorem, a foundational result in general relativity and differential geometry
\cite{MR535700,MR612249} \cite{MR626707} (cf. \cite[Discussion after Theorem 1.7] {Wang:2021tq}).

	Before we state our main results, let us first recall Gromov's dihedral extremality and dihedral rigidity conjectures for convex Euclidean polyhedra.   Given a Riemannian metric $g$ on an oriented manifold $M$ with polyhedral boundary (cf. Definition \ref{def:polytopeboundary}), we shall denote the scalar curvature of $g$ by $\Sc(g)$, the mean curvature\footnote{Our sign convention for the mean curvature is that the mean curvature of the standard round sphere viewed as the boundary of a Euclidean ball is positive. } of each face $F_i$ of $M$ by $H_g(F_i)$, and the dihedral angle function of two adjacent faces $F_i$ and $F_j$ by $\theta_{ij}(g)$. Here the dihedral angle $\theta_{ij}(g)_x$ at a point $x\in F_i\cap F_j$ is defined as follows. 
	\begin{definition}\label{def:diheral}
		Write $F_{ij} = F_i\cap F_j$. Let $u$ and $v$ be the unit inner normal vector of $F_{ij}$ with respect to $F_i$ and $F_j$ at $x\in F_{ij}$, respectively.  Let $\theta_{ij}(g)_x$ be either the angle of $u$ and $v$, or $\pi$ plus this angle, depending on the vector $(u+v)/2$ points inward or outward, respectively. See Figure \ref{fig:dihedralangles}.
	\end{definition} 
	
	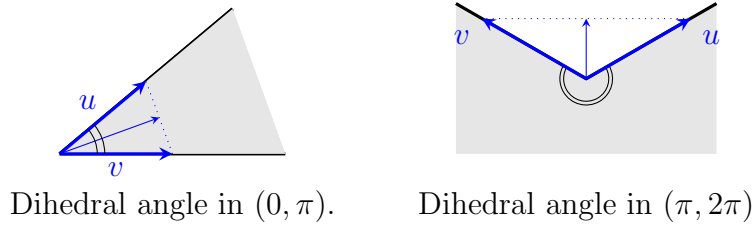
\begin{figure}
		\begin{tikzpicture}[scale=1]
			\fill[gray!20] ({7-sqrt(3)},2) -- (7,1) -- ({7+sqrt(3)},2) -- ({7+sqrt(3)},0) -- ({7-sqrt(3)},0) -- cycle;
			\fill[gray!20] (0,0) -- (3,0) -- ({3*sin(50)},{3*cos(50)}) -- (0,0);
			\draw[very thick,black] (0,0) -- (3,0);
			\draw[very thick,black] (0,0) -- ({3*sin(50)},{3*cos(50)});
			\filldraw (1.5,-0.7) node {Dihedral angle in $(0,\pi)$.};
			\filldraw (7,-0.7) node {Dihedral angle in $(\pi,2\pi)$};
			\draw[very thick,black] ({7-sqrt(3)},2) -- (7,1) -- ({7+sqrt(3)},2) ;
			\draw[very thick,black] (0,0) -- ({3*sin(50)},{3*cos(50)});
			\fill[gray!20] (0,0) -- (3,0) -- ({3*sin(50)},{3*cos(50)}) -- (0,0);
			\draw[very thick, blue,-stealth] (0,0) -- (1.5,0);
			\draw[very thick, blue,-stealth] (0,0) -- ({1.5*sin(50)},{1.5*cos(50)});
			\draw[blue,-stealth] (0,0) -- ({(1.5*sin(50)+1.5)/2},{1.5*cos(50)/2});
			\draw[blue,dotted]({1.5*sin(50)},{1.5*cos(50)}) -- (1.5,0);
			\draw[black] (0.5,0) arc (0:40:0.5);
			\draw[blue] ({1.5*sin(50)/2-0.2},{1.5*cos(50)/2+0.2}) node {$u$};
			\draw[blue] ({1.5/2},-0.2) node {$v$};
			\draw[black] (0.6,0) arc (0:40:0.6);
			\draw[black] ({7+0.15*sqrt(3)},{1+0.15}) arc (30:-210:0.3);
			\draw[black] ({7+0.175*sqrt(3)},{1+0.175}) arc (30:-210:0.35);
			\draw[very thick, blue,-stealth] (7,1) -- ({7+sqrt(3)*0.8},1.8) node[anchor=north west] {$u$};
			\draw[very thick, blue,-stealth] (7,1) -- ({7-sqrt(3)*0.8},1.8) node[anchor=north east] {$v$};
			\draw[blue,-stealth] (7,1) -- ({7},1.8) ;
			\draw[blue,dotted]({7+sqrt(3)*0.8},1.8) -- ({7-sqrt(3)*0.8},1.8);
		\end{tikzpicture}
		\label{fig:dihedralangles}
		\caption{Dihedral angles.}
	\end{figure}
	Here the angle $\theta_{ij}(g)_x$ takes value in $(0,\pi)\cup(\pi,2\pi)$. Roughly speaking, if $M$ is convex at $x$, then $\theta_{ij}(g)_x<\pi $;  and if $M$ is concave at $x$, then  $\theta_{ij}(g)_x>\pi $. Furthermore, our sign convention for the mean curvature is that the mean curvature of the standard round sphere viewed as the boundary of a Euclidean ball is positive. 	
	
\begin{conjecture}[{Gromov's dihedral extremality conjecture for convex polyhedra, \cite[Section 7]{Gromovadozen}}]\label{conj:extremal}
	Let $P$ be a convex polyhedron in $\R^n$ and $g$ the Euclidean metric on $P$. If $\overbar g$ is a smooth Riemannian metric on $P$ such that 
	\begin{enumerate}[label=$(\arabic*)$]
		\item $\Sc(\overbar g)\geq \Sc(g) = 0$,
		\item $H_{\overbar g }(F_i)\geq H_{g}(F_i) = 0$ for each face $F_i$ of $P$, and
		\item $\theta_{ij}(\overbar g)\leq \theta_{ij}(g)$ on each $F_{ij} = F_i\cap F_j$,	
	\end{enumerate}
	then we have 
	\[ \Sc(\overbar g)=0, H_{\overbar g}(F_i) = 0 \textup{ and }  \theta_{ij}(\overbar g) =  \theta_{ij}(g)\]
	for all $i$ and for all $i\neq j$. 
\end{conjecture}

\begin{conjecture}[{Gromov's dihedral rigidity conjecture for convex polyhedra, \cite[Section 2.2]{GromovDiracandPlateau}}]\label{conj:dihedral}
		Let $P$ be a convex polyhedron in $\R^n$ and $g$ the Euclidean metric on $P$. If $\overbar g$ is a smooth Riemannian metric on $P$ such that 
	\begin{enumerate}[label=$(\arabic*)$]
		\item $\Sc(\overbar g)\geq \Sc(g) = 0$,
		\item $H_{\overbar g }(F_i)\geq H_{g}(F_i) = 0$ for each face $F_i$ of $P$, and
		\item $\theta_{ij}(\overbar g)\leq \theta_{ij}(g)$ on each $F_{ij} = F_i\cap F_j$,	
	\end{enumerate}
	then $\overbar g$ is also a flat metric. 
\end{conjecture} 

Gromov's flat corner domination conjecture is an even stronger conjecture on convex Euclidean polyhedra  \cite[Section 3.18]{Gromov4lectures2019}. 
\begin{conjecture}[Gromov's flat corner domination conjecture]\label{conj:flatcorner}
	Let $P$ be a convex polyhedron in $\R^n$ and $g$ the Euclidean metric on $P$. If $\overbar g$ is a smooth Riemannian metric on $P$ such that 
	\begin{enumerate}[label=$(\arabic*)$]
		\item $\Sc(\overbar g)\geq \Sc(g) = 0$,
		\item $H_{\overbar g }(F_i)\geq H_{g}(F_i) = 0$ for each face $F_i$ of $P$, and
		\item $\theta_{ij}(\overbar g)\leq \theta_{ij}(g)$ on each $F_{ij} = F_i\cap F_j$,	
	\end{enumerate}
	then $\overbar g$ is flat and  all codimension one faces of $(P, \overbar g)$ are flat; moreover, at every point $x\in P$, the manifold $(P, \overbar g)$ is locally isometric to $(P, g)$. 
\end{conjecture} 
We emphasize that  Gromov's flat corner domination conjecture not only determines the geometry of the interior of $(P, \overbar g)$ (which has to be flat), but also determines  the geometry of all faces (of any codimension) and all angles (not necessarily dihedral angles) between  faces (of any codimension). The latter implication will be the key step for solving the Stoker conjecture.  

Our first  main result of the paper is a positive solution to the above flat corner domination conjecture in all dimensions. 

\begin{theorem}\label{thm:polyhedra}
	Let $P$ be a convex polyhedron in $\R^n$ and $g$ the Euclidean metric on $P$. If $\overbar g$ is a smooth Riemannian metric on $P$ such that 
	\begin{enumerate}[label=$(\arabic*)$]
		\item $\Sc(\overbar g)\geq \Sc(g) = 0$,
		\item $H_{\overbar g }(F_i)\geq H_{g}(F_i) = 0$ for each face $F_i$ of $P$, and
		\item $\theta_{ij}(\overbar g)\leq \theta_{ij}(g)$ on each $F_{ij} = F_i\cap F_j$,	
	\end{enumerate}
	then $\overbar g$ is flat and  all codimension one faces of $(P, \overbar g)$ are flat; moreover, at every point $x\in P$, the manifold $(P, \overbar g)$ is locally isometric to $(P, g)$. 
\end{theorem}

In our previous joint paper with Yu \cite{Wang:2021tq}, the authors completely settled  Gromov's dihedral extremality conjecture for convex polyhedra   in all dimensions \cite[Theorem 1.8]{Wang:2021tq} by developing  a new index theory for manifolds with polyhedral boundary.   Moreover, the same techniques  in \cite{Wang:2021tq} also imply  Gromov's dihedral rigidity conjecture for convex polyhedra in dimension three \cite[Theorem 1.8]{Wang:2021tq}. While the techniques  in \cite{Wang:2021tq} fell short of proving the dihedral rigidity conjecture in dimension $\geq 4$, we shall  apply new index theoretic methods in the present paper to solve Gromov's dihedral rigidity   conjecture in all dimensions.  In fact, our methods lead to a positive solution to Gromov's flat corner domination conjecture in all dimensions.  More precisely, we prove the general form of  Gromov's flat corner domination  conjecture  that actually  allows comparisons of possibly  different manifolds \cite[Section 3.18]{Gromov4lectures2019}. See Theorem \ref{thm:strongflat} below and its variant Theorem \ref{thm:flatWithVertex}  for the precise details. We then deduce Theorem \ref{thm:polyhedra}  as a special case of either Theorem \ref{thm:strongflat} or Theorem \ref{thm:flatWithVertex}.

As a consequence of Theorem \ref{thm:polyhedra}, we answer positively the Stoker conjecture for convex Euclidean polyhedra in all dimensions \cite{MR222765}.  More precisely, we have the following theorem. 

\begin{theorem}\label{thm:stokereuclidean}
If $P_1$ and $P_2$ are two convex Euclidean polyhedra of the same combinatorial type such that all corresponding dihedral angles are equal, then all corresponding face angles\footnote{Here the face angles refer to the dihedral angles of each codimension one face (thought of as a polyhedron itself).}  are equal. 
\end{theorem} 

 Let us mention briefly some of the  previous work on the Stoker conjecture. There have been many attempts to solve the Stoker conjecture in the past fifty years. The conjecture was verified in some special cases. For example,  Karcher verified the Stoker conjecture for a class of $3$-dimensional convex polyhedra with 5 vertices and 6 faces \cite{MR222766}. There is also an analogous conjecture for convex hyperbolic polyhedra, which has also been known in some special cases. For example, Andreev \cite{MR0259734} proved the hyperbolic Stoker conjecture for convex hyperbolic polyhedra with all dihedral angles less than $\pi/2$.  Mazzeo and Montcouquiol proved a weaker version (an \emph{infinitesimal} version) of the Stoker conjecture  \cite[Theorem 1]{MR2819548}. We refer the reader to \cite[Theorem 1]{MR2819548} for the precise statement of this infinitesimal version of the Stoker conjecture.  We should also mention that the analogue of Stoker's conjecture for convex spherical polyhedra is false, due to counterexamples of Schlenker \cite{MR1744513}. 

So far, we have been mainly concerned with convex polyhedra in Euclidean spaces. In fact, we can apply the same methods of the present paper to prove similar rigidity results for manifolds with smooth boundary. For example,  we have the following rigidity theorem for strictly convex domains with smooth boundary in Euclidean spaces.

\begin{theorem}\label{thm:balliso}
	Let $(M,g)$ be a strictly convex domain  with smooth boundary in $\R^n$ \textup{($n\geq 2$)}. Let $(N,\overbar g)$ be a spin Riemannian manifold with boundary and $f\colon N\to M$ be a spin map. If
	\begin{enumerate}[label=$(\arabic*)$]
		\item 
		$\Sc(\overbar{g})_x \geq \Sc(g)_{f(x)} = 0$ for all $x\in N$, 
		\item $H_{\overbar{g}}(\partial N)_y \geq  H_{g}(\partial M)_{f(y)}$ for all $y\in \partial N$, 
		\item $f$ is distance-non-increasing on $N$, 
		\item the degree of $f$ is nonzero,	
	\end{enumerate} 
	then $f$ is an isometry. 
\end{theorem}
   Here $f\colon N \to M$ is said to be distance-non-increasing at  $x\in N$ if $\|df\|_x\leq 1$,  
where $df\colon TN \to TM$ is the tangent map.  As a special case of  the above theorem, we see that,  given a strictly convex domain with smooth boundary in $\mathbb R^n$, one cannot  increase the metric,  the scalar curvature and the mean curvature of its boundary simultaneously.  If we relax the condition ``$f$ is distance-non-increasing on $N$'' to ``$f$ is distance-non-increasing on $\partial N$'', then one can still conclude that $(N, \overbar g)$ is isometric to $(M, g)$ as smooth manifolds with boundary (cf. Theorem \ref{thm:ballAlmostIso} below). However, under such a weaker assumption, $f$ itself may not be an isometry. To obtain the stronger conclusion that $f$ is itself an isometry, one generally needs to require  $f$ to be distance-non-increasing on the whole $N$. In fact, as elementary examples show (cf. the discussion before the proof of Theorem \ref{thm:balliso} in Section \ref{sec:smooth}), there exists a degree one map $f\colon (\mathbb B^2, g_{st}) \to (\mathbb B^2, g_{st})$ such that $f$ is area-non-increasing\footnote{Here $f\colon N \to M$ is said to be area-non-increasing at  $x\in N$ if $\|\wedge^2df\|_x\leq 1$,  
	where $\wedge^2 df\colon \Bigwedge^2 TN \to \Bigwedge^2 TM$ is the map on two forms.} on $\mathbb B^2$ and $f$ equals the identity map  on $\partial \mathbb B^2$, but $f$ is not an isometry.  Here $\mathbb B^2$ is the standard unit Euclidean ball equipped with the standard Euclidean metric $g_{st}$. 

The paper is organized as follows. In Section \ref{sec:dihedral}, we introduce  a  notion of manifolds with polyhedral boundary, which is a class of manifolds that includes for example all polyhedra.    We prove the general form of Gromov's  flat corner domination conjecture for all flat manifolds with polyhedral boundary in all dimensions.  As a special case, we answer positively   Gromov's flat corner domination conjecture for convex Euclidean polyhedra in all dimensions (Theorem \ref{thm:polyhedra}). Consequently,  we obtain a positive solution to the Stoker conjecture for convex Euclidean polyhedra (Theorem \ref{thm:stokereuclidean}). In Section \ref{sec:smooth}, we prove a scalar-mean rigidity theorem for strictly convex Euclidean  domains (Theorem \ref{thm:balliso}). 

\vspace{.5cm}
\textbf{Acknowledgments.} We would like to thank Tian Yang and Bo Zhu for helpful comments. 

\section{Dihedral Rigidity of flat manifolds}
\label{sec:dihedral}
In this section, we prove Gromov's flat corner domination  conjecture for convex Euclidean polyhedra (Theorem \ref{thm:polyhedra}) and the Stoker conjecture for convex Euclidean polyhedra (Theorem \ref{thm:stokereuclidean}). 
\subsection{Manifolds with polyhedral boundary}
In this subsection, we introduce a notion of manifolds with polyhedral boundary. We also review the index theory on manifolds with polyhedral boundary  developed in \cite{Wang:2021tq}.

Recall that $n$-dimensional smooth manifolds with corners are locally modeled on $[0, \infty)^k\times \mathbb R^{n-k}$ with $0\leq k \leq n$. More precisely, let $M$ be a Hausdorff space. A chart $(U, \varphi)$ (possibly with corners) for $M$ is a homeomorphism $\varphi$ from an open subset $U$ of $M$ to an open subset of  $[0, \infty)^k\times \mathbb R^{n-k}$ for some $0\leq k \leq n$. Two charts $(U_1, \varphi_1)$ and $(U_2, \varphi_2)$ are $C^\infty$-related if either $U_1\cap U_2$ is empty or the map 
\[ 	\varphi_2\circ \varphi_1^{-1}\colon \varphi_1(U_1\cap U_2) \to \varphi_2(U_1\cap U_2) \]
is a diffeomorphism (of open subsets in $[0, \infty)^{k_1}\times \mathbb R^{n-k_1}$ and $[0, \infty)^{k_2}\times \mathbb R^{n-k_2}$). A system of pairwise $C^\infty$-related charts of $M$ that covers $M$ is called an atlas of $M$. A smooth manifold with corners is a Hausdorff space equipped with a maximal atlas of charts. 

Similarly, we introduce the following notion of manifolds with polyhedral  boundary, which are locally modeled on $n$-dimensional polyhedra in $\mathbb R^n$. For a given Hausdorff space $X$, a polytope chart $(U, \varphi)$  for $X$ is a homeomorphism $\varphi$ from an open subset $U$ of $M$ to an open subset of an $n$-dimensional polyhedron in $\mathbb R^n$.  Two polytope charts $(U_1, \varphi_1)$ and $(U_2, \varphi_2)$ are $C^\infty$-related if either $U_1\cap U_2$ is empty or the map 
\[ 	\varphi_2\circ \varphi_1^{-1}\colon \varphi_1(U_1\cap U_2) \to \varphi_2(U_1\cap U_2) \]
is a diffeomorphism (of open subsets of $n$-dimensional polyhedra). Again, a system of pairwise $C^\infty$-related charts of $X$ that covers $X$ is called an atlas of $X$. 
\begin{definition}\label{def:polytopeboundary}
	A smooth manifold with polyhedral  boundary is a Hausdorff space equipped with a maximal atlas of polytope charts. 
\end{definition}

A Riemannian manifold with polyhedral boundary is a smooth manifold with polyhedral boundary equipped with a smooth Riemannian metric.  A main difference between manifolds with corners and manifolds with polyhedral boundary is the following: for an $n$-dimensional manifold with corners, there can be at most $n$ codimension one faces meeting at a given point; while there may be more than $n$ codimension one faces meeting at a given point in an $n$-dimensional manifold with polyhedral boundary.

\begin{definition}
	Let $N$ be an $n$-dimensional manifold with polyhedral boundary. We define the codimension $k$ stratum of $N$ to be the set of interior points of all codimension $k$ faces of $N$. 
\end{definition}

For each point $x$ in the codimension $k$ stratum of $N$, it admits a small neighborhood $U$ of the form: 
\[ \R^{n-k} \times P\]
such that $P$ is a polyhedral corner in $\R^k$ enclosed by hyperplanes passing through the origin of $\R^k$ and $x$ is the origin of $\R^n$.  In this case, we call the partial derivatives along $\mathbb R^{n-k}$ the base directions of the neighborhood $U$ of $x$.

\begin{definition}\label{def:polytopeMap} A  map $f\colon (N,\overbar g)\to (M,g)$ between Riemannian manifolds with polyhedral boundary is called a \emph{polytope map} if 
	\begin{enumerate}
		\item 	$f$ is Lipschitz\footnote{Here $f\colon (N,\overbar g)\to (M,g)$ is said to be Lipschitz if there exists $C>0$ such that $d_M(f(x), f(y))\leq C\cdot d_N(x, y)$  for all points in $x, y\in N$.},
		\item $f$ is smooth away from the codimension three faces of $N$,
		\item $f$ maps the codimension $k$ stratum of $N$ to the  codimension $k$ stratum of $M$, and 
		\item every point $x$ in $N$ has a small open neighborhood $U$ such that $f$ is smooth with respect to the base directions on $U$.  
	\end{enumerate}
\end{definition}

\begin{remark}
	Condition (4) in the above definition of polytope maps is mainly added for technical reasons. It can certainly be weakened without affecting all the main results in the present paper. However, imposing condition (4) makes some of the proofs of this paper a little more transparent. In any case,  condition (4) is always satisfied  in the main geometric applications that we are concerned with. 
\end{remark}

We emphasize that a polytope map $f\colon N \to M$ is \emph{not} required to be smooth at the codimension three faces of $N$. Such a flexibility will be important when we consider $n$-dimensional polyhedral corners that have more than $n$ codimension one faces meeting at their vertices (e.g. when we prove the Stoker conjecture in Theorem \ref{thm:stokereuclidean}).  For example, let $N$ and $M$ be two convex polyhedra in $\R^n$ with the same combinatorial type. Then there is always a smooth polytope map $f\colon N\to M$. One can construct $f$ smoothly near each codimension $2$ edge of $N$, and inductively extend $f$ radially towards   higher codimension vertices.

Consider the vector bundle $f^*TM$ over $N$, which is equipped with the pull-back connection $f^*\nabla^M$ of the Levi--Civita connection on $M$. The smooth structure of $f^*TM$ is defined everywhere away from faces of codimension $\geq 2$. In particular, it makes sense to talk about the space of smooth sections of $f^\ast TM$ that vanishes near codimension two faces, which will be denoted by  $C^\infty_0(N,f^*TM)$. Moreover, the connection $f^*\nabla^M$ is well-defined away from codimension two  faces.

 We define
$H^1(N,f^*TM)$ to be the completion of $C^\infty_0(N,f^*TM)$ with respect to the the following Sobolev $H^1$-norm:
\begin{equation}\label{eq:sobolevh1} 
	\|s\|_1\coloneqq \big(\|s\|^2+\|\widetilde \nabla s\|^2\big)^{1/2}
\end{equation}
for $s\in C^\infty_0(N,f^*TM)$, where $\widetilde \nabla = f^\ast \nabla^M$. 
\begin{lemma}\label{lemma:H1space}
	The space $H^1(N,f^*TM)$  is independent (up to bounded isomorphisms of Hilbert spaces) of the metric on $TM$, and coincides with the usual $H^1$-space if $f$ is smooth.
\end{lemma}
\begin{proof}
	Let  $\{U_\alpha\}$ be an open cover of $N$ consisting of polytope charts such that  $TN$ is trivial on $U_\alpha$ and $TM$ is trivial on $f(U_\alpha)$. Let $\{\phi_\alpha\}$ be a smooth partition of unity subordinate to $\{U_\alpha\}$. Set $m = \dim M$. For each $s\in C^\infty_0(N,f^*TM)$,  we may view $s_\alpha\coloneqq \varphi_\alpha s$ with a smooth function from $N$ to $\R^m$ after we identify $f^\ast TM $ with a trivial bundle over $N$. More precisely, we choose a smooth orthonormal basis $\{e_i\}$ of $TM$ over $f(U_\alpha)$. Then $s_\alpha$ is uniquely written as
	$$s_\alpha=\sum_{i=1}^n s^i_\alpha e_i$$
	where $s^i_\alpha$'s  are smooth functions vanishes near codimension two  faces.
	
	Let $\Gamma_{ij}^k$ be the Christoffel symbols of the Levi-Civita connection on $TM$, that is, 
	$$\nabla^M_{e_i}e_j=\sum_{k=1}^n\Gamma_{ij}^ke_k.$$
	Let $\{X_1,\cdots,X_n\}$ be a  orthonormal basis of $TN$ over $U_\alpha$. Then we may write  
	$$f_*X_i=\sum_{j=1}^n x^j_ie_j,$$
	where $x_i^j$'s are bounded functions over $M$ (since $f$ is Lipschitz) and  smooth in the interior of $M$ (since $f$ is smooth away from codimension two faces). It follows that 
	$$\widetilde \nabla_{X_i}s_{\alpha}=\sum_{j=1}^nX_i(s_\alpha^j) \cdot e_j+\sum_{j,k=1}^n x_i^j\Gamma_{ji}^k e_k.$$
	Since $x_i^j$ and $\Gamma_{ji}^k$ are uniformly bounded over $U_\alpha$, it is  not difficult to see that the  Sobolev $H^1$ norm from line \eqref{eq:sobolevh1} is equivalent to the following norm: 
	$$\|s_\alpha\|_{new}^2\coloneqq \sum_{i=1}^n\|s_\alpha^i\|^2+\sum_{i=1}^n\|\grad(s_\alpha^i)\|^2,$$
	where the latter is independent of the choice of the metric on $M$. Together with  the partition of unity $\{\phi\alpha\}$ subordinate to $\{U_\alpha\}$, it follows that $H^1(N,f^*TM)$ is independent of the metric on $TM$, up to bounded isomorphisms of Hilbert spaces.
	
	Now assume $f$ is smooth. Then $f^*TM$ is a smooth vector bundle over the entire $N$. Recall that  removing a subspace of codimension $\geq 2$ does not affect the definition of Sobolev $H^1$ spaces. In particular,   the space of smooth sections that vanish near codimension two faces of $N$ is dense in the usual $H^1$-space, where the usual  $H^1$-space is the completion of all smooth sections over $N$ (that do not necessarily vanish near codimension two faces of $N$). This finishes the proof.
\end{proof}

The proof of Lemma \ref{lemma:H1space} shows that locally the $H^1$-norm from \eqref{eq:sobolevh1} is equivalent to the usual $H^1$-norm of vector-valued functions. In particular, it follows that the inclusion $H^1(N,f^*TM) \to L^2(N,f^*TM)$ is an compact operator. Moreover, there is a bounded trace map $H^1(N,f^*TM) \to H^{1/2}(N, f^\ast TM)$.

As we will work with Dirac type operators, let us recall the following definition of spin maps. 
\begin{definition}\label{def:spin}
	A map $f\colon N \to M$ is said to be  a spin map if the second Stiefel--Whitney classes of $TM$ and $TN$ are related by 
	\[  w_2(TN) = f^\ast(w_2(TM)).\]
	Equivalently, $f\colon N \to M$ is a spin map   if $TN\oplus f^\ast TM$ admits a spin structure.
\end{definition}

Now we assume that both $N$ and $M$ are even dimensional Riemannian manifolds with polyhedral  boundary and $f\colon N\to M$ is a spin polytope map. The odd dimensional case is completely similar, or alternatively may be reduced to the even dimensional case by taking the direct product with the unit interval.   The Riemannian metric on $TM$ pulls back to a (continuous) Riemannian metric on $f^\ast TM$ over $N$. In particular, the bundle $TN\otimes f^*TM$ over $N$ admits a natural Riemannian metric.  Let $S_N\otimes f^*S_M$ be the spinor bundle of $TN\otimes f^*TM$, which exists as $f$ is a spin map. Let  $\nabla$ be the associated Riemannian spinor connection on $S_N\otimes f^*S_M$, which is well-defined at least away from the codimension two faces of $N$. More precisely, away from codimension two faces of $N$, we have
$$\nabla=\nabla^{S_N}\otimes 1+1\otimes f^*(\nabla^{S_M}),$$
where $\nabla^{S_N}$ and $\nabla^{S_M}$ are the Levi--Civita connection on $S_N$ and $S_M$, respectively. 
Let $D$ be the Dirac operator on $S_N\otimes f^*S_M$ given by
$$D=\sum_{i=1}^n\overbar c(\overbar e_i)\nabla_{\overbar e_i},$$
where $\{\overbar e_i\}$ is a local orthonormal frame of $TN$. Clearly, $D$ is well-defined away from the codimension two faces of $N$. Note that a small open neighborhood of a point $x$ in the interior of  a codimension $k$ face  is homeomorphic to a fiber bundle $W\times \fiber$, where $\fiber$ is  some Euclidean polyhedral corner.  Since $f$ is a polytope map, we see that $D$ is fiberwise asymptotically conical, that is, the fiberwise Dirac operator $D^\fiber$ along each fiber $\fiber_y$ is asymptotically conical, cf. \cite[Section 3]{Wang:2021tq}. Also see  \cite[RS4]{BruningSeeley} and \cite[Section 1]{Bruning}.

Note that $\nabla$ on $S_N \otimes f^\ast S_M$ is the spinorial connection induced by the connection $\nabla^M\oplus f^*(\nabla^M)$ on $TN\oplus f^*TM$. By Lemma \ref{lemma:H1space}, the associated $H^1$-space is also well-defined and  (up to bounded isomorphisms) independent of the metric on $M$. 

 Let $\overbar \epsilon$ and $\epsilon$ are the grading operators of $S_N$ and $f^\ast S_M$ respectively. Let $B$ be the boundary condition on $S_N\otimes f^*S_M$ over each codimension one face $\overbar F_i$ of $N$ given by
\begin{equation}\label{eq:boundarycondition}
	(\overbar\epsilon\otimes\epsilon)(\overbar c(\overbar e_n)\otimes c(e_n))\varphi=-\varphi
\end{equation}
for all smooth sections $\varphi$ of $S_N\otimes f^\ast S_M$ over $N$, where $\overbar e_n$ is the unit inner normal vector of $\overbar F_i$ in $N$ and $e_n$ is the unit inner normal vector of the corresponding face  $F_i$ in $M$. 
\begin{definition}\label{def:smooth,H1} \mbox{}
	\begin{enumerate}
		\item Let $C^\infty_0(N,S_N\otimes f^*S_M;B)$ be the collection of smooth sections that satisfies the boundary condition $B$ at each codimension one face and vanishes near all faces with codimension $\geq 2$.
		\item Let $H^1(N,S_N\otimes f^*S_M;B)$ be the completion of  $C^\infty_0(N,S_N\otimes f^*S_M;B)$ with respect to  the $H^1$-norm
		$$\|\varphi\|_1\coloneqq (\|\varphi\|^2+\|\nabla\varphi\|^2)^{1/2}.$$
	\end{enumerate}
\end{definition}

One of the main results of  \cite{Wang:2021tq} is the following theorem.   
\begin{theorem}[{\cite[Theorem 1.6]{Wang:2021tq}}]\label{thm:previous}
	Suppose $(M,g)$ is a compact oriented $n$-dimensional submanifold with polyhedral boundary such that
\begin{enumerate}[label=$(\alph*)$]
	\item $M$ has nonzero Euler characteristic,
	\item the curvature form of $M$ is non negative, and each of its codimension one face has non-negative second fundamental form, and 
	\item all of its dihedral angles are $ < \pi$.
\end{enumerate}   
Let $(N, \overbar{g})$ be a compact oriented $n$-dimensional spin manifold with polyhedral boundary. If $f\colon (N, \overbar{g}) \to (M, g)$  is a spin polytope map such that	
\begin{enumerate}[label=$(\arabic*)$]
	\item  $f$ is area-non-increasing on $N$  and distance-non-increasing on $\partial N$, 
	\item the degree of $f$ is nonzero,	
	\item 
	$\Sc(\overbar{g})_x \geq \Sc(g)_{f(x)}$ for all $x\in N$, 
	\item $H_{\overbar{g}}(\overbar{F}_i)_y \geq  H_{g}(F_i)_{f(y)}$ for all $y$ in each codimension one face\footnote{The notation $\overbar{F}_i$ and $F_i$ means that the map $f$ takes the face $\overbar{F}_i$ of $N$ to the face $F_i$ of $M$.} $\overbar{F}_i$ of $N$, and
	\item $\theta_{ij}(\overbar{g})_z\leq \theta_{ij}(g)_{f(z)}$ for all $\overbar F_i, \overbar F_j$ and all $z \in \overbar F_i\cap \overbar{F}_j$,   
\end{enumerate}
then we have 
 \begin{enumerate}
	\item 
   $\Sc(\overbar{g})_x = \Sc(g)_{f(x)} = 0$ for all $x\in N$, 
   \item $H_{\overbar{g}}(\overbar{F}_i)_y =  H_{g}(F_i)_{f(y)}$ for all $y$ in each codimension one face $\overbar{F}_i$ of $N$, 
   \item $\theta_{ij}(\overbar{g})_z =  \theta_{ij}(g)_{f(z)}$ for all $\overbar F_i, \overbar F_j$ and all $z \in \overbar F_i\cap \overbar{F}_j$. 
\end{enumerate}
Furthermore,\footnote{For the last conclusion, see Claim 7.1 in the proof of \cite[Theorem 1.7]{Wang:2021tq}.} there exists a non-zero parallel section $\varphi\in H^1(N,S_N\otimes f^*S_M;B)$, i.e., $\nabla\varphi=0$.
\end{theorem}

\subsection{Gromov's flat corner domination conjecture}\label{sec:main-sub}
In this subsection, we prove  the following theorem, which answers positively  Gromov's flat corner domination conjecture \cite[Section 3.18]{Gromov4lectures2019}.
\begin{theorem}\label{thm:strongflat}
	Suppose $M$ is a compact oriented $n$-dimensional flat submanifold with polyhedral boundary in $\R^n$ with the flat metric $g$  such that
	\begin{enumerate}[label=$(\alph*)$]
		\item $M$ has nonzero Euler characteristic,
		\item each of its codimension one face is convex, that is, its second fundamental form is non-negative, 
		\item all of its dihedral angles are $ < \pi$,
	\end{enumerate}   
Let $(N, \overbar{g})$ be a compact oriented $n$-dimensional spin manifold with polyhedral boundary. If $f\colon (N, \overbar{g}) \to (M, g)$  is a  polytope  map such that	
\begin{enumerate}[label=$(\arabic*)$]
		\item  $f$ is distance-non-increasing on $\partial N$, 
	\item the degree of $f$ is nonzero,	
	\item 
	$\Sc(\overbar{g})_x \geq \Sc(g)_{f(x)} = 0$ for all $x\in N$, 
	\item $H_{\overbar{g}}(\overbar{F}_i)_y \geq  H_{g}(F_i)_{f(y)}$ for all $y$ in each codimension one face $\overbar{F}_i$ of $N$, 
	\item $\theta_{ij}(\overbar{g})_z\leq \theta_{ij}(g)_{f(z)}$ for all $\overbar F_i, \overbar F_j$ and all $z \in \overbar F_i\cap \overbar{F}_j$,   
\end{enumerate}
then 
\begin{enumerate}
	 \item 
	$\Sc(\overbar{g})_x = \Sc(g)_{f(x)} = 0$ for all $x\in N$, 
	\item $H_{\overbar{g}}(\overbar{F}_i)_y =  H_{g}(F_i)_{f(y)}$ for all $y$ in each codimension one face $\overbar{F}_i$ of $N$, 
	\item $\theta_{ij}(\overbar{g})_z =  \theta_{ij}(g)_{f(z)}$ for all $\overbar F_i, \overbar F_j$ and all $z \in \overbar F_i\cap \overbar{F}_j$,
\end{enumerate} 
and $(N, \overbar g)$ is also flat. Furthermore, the following hold. 
\begin{enumerate}[label=$(\roman*)$]
	\item  If a codimension one face $F_i$ of $M$ is flat, then the corresponding face $\overbar F_i$ of $N$ is also flat.
	\item Suppose that $\overbar x$ is a point in the intersection of $\ell$ codimension one faces of $N$ whose unit inner normal vectors at $x$ are denoted by $\overbar \nu_1,\ldots,\overbar \nu_\ell$. Let $\nu_1, \ldots, \nu_\ell$ be the unit inner normal vectors at $f(\overbar x)$ of the corresponding faces of $M$. Then we have 
	\begin{equation}\label{eq:linear}
\langle \nu_i, \nu_j\rangle_g=\langle \overbar \nu_i,\overbar \nu_j\rangle_{\overbar g},~\forall i,j=1,\ldots,\ell.
	\end{equation} 
\item Suppose that $\overbar x_1,\overbar x_2\in \partial N$ and $f(\overbar x_1),f(\overbar x_2)\in \partial M$ are in the interior of codimension one faces. Denote their corresponding unit inner normal vectors  by $\nu_{\overbar x_1}$, $\nu_{\overbar x_2}$, $\nu_{f(\overbar x_1)}$ and $\nu_{f(\overbar x_2)}$. Then we have
$$\langle \widetilde{\nu}_{\overbar x_1},\nu_{\overbar x_2}\rangle=\langle \widetilde \nu_{f(\overbar x_1)},\nu_{f(\overbar x_2)}\rangle,$$
where $\widetilde{\nu}_{\overbar x_1}$ is the parallel transport of $\nu_{\overbar x_1}$ from $\overbar x_1$ to $\overbar x_2$ along any piecewise smooth path, and $\widetilde{\nu}_{f(\overbar x_1)}$ is the parallel transport\footnote{In fact, we have $ \widetilde{\nu}_{f(\overbar x_1)} =\nu_{f(\overbar x_1)} \in \mathbb R^n$, since $(M, g)$ is assumed to be a submanifold of $\mathbb R^n$.} of $\nu_{f(\overbar x_1)}$ from $f(\overbar x_1)$ to $f(\overbar x_2)$ along any piecewise smooth path.
\end{enumerate}

\end{theorem}

\begin{remark}
    The conclusion in the case of manifolds with polyhedral boundary carries is geometrically more significant than  the case of manifolds with corners.  Let us consider  part $(ii)$ of the above theorem for example. For  an $n$-dimensional manifold $Y$ with corners,  each vertex  of $Y$ is the intersection of precisely $n$ codimension one faces. These faces pairwise intersect, hence the equalities in line \eqref{eq:linear} simply become the equalities of corresponding dihedral angles in the case of manifolds with corners.  However, if $N$ is an $n$-dimensional manifold with polyhedral boundary, then it is possible for a point $x\in N$ to lie  in the intersection of many more than $n$ codimension one faces of $N$, where $n = \dim N$. In this case,  the equalities in line \eqref{eq:linear} imply that not only the equalities of  corresponding dihedral angles, but also the equalities of all other corresponding angles between non-adjacent faces.  
\end{remark}
\begin{remark}
	It will be clear from the proof  that the conclusion in part $(ii)$ and part $(iii)$ of Theorem \ref{thm:strongflat} also hold under the assumption  of Theorem \ref{thm:previous}.
	\end{remark}

\begin{remark}
We point out that, for a given codimension one face $\overbar F$ of $N$, if the corresponding face $F$ of $M$ is flat,  then  the condition (1)  for requiring $f$ to be distance-non-increasing  on $\overbar F$ is \emph{not} needed. This follows from a standard estimate involving the second fundamental form of $F$. See  for example  \cite[Lemma 2.3]{Wang:2021tq}. Also see the proof of Proposition \ref{prop:strictlyConvex} for a similar computation. 
\end{remark}

\begin{proof}[Proof of Theorem \ref{thm:strongflat}]
The odd dimensional case can easily be reduced to the even dimensional case by considering $N\times [0,1]$ and $M\times[0,1]$ with the obvious product metrics. Therefore, we can assume without loss of generality  that $N$ and $M$ are even-dimensional.

Let $S_N$ and $S_M$ be the corresponding spinor bundle over $N$ and $M$. By Theorem \ref{thm:previous}, there exists a non-zero section $\varphi\in H^1(N,S_N\otimes f^*S_M;B)$ such that $\nabla\varphi=0$. It follows that $\varphi$ is smooth away from codimension $2$ faces. Moreover,  $\varphi$ extends continuously everywhere over $N$ in the following sense. For each $x\in N$ on a face with codimension $\geq 2$, consider a local chart near $x$ so that sections of $S_N\otimes f^*S_M$ near $x$ is identified with a vector in $\R^{2^n}$, cf. Lemma \ref{lemma:H1space}. Let $\gamma$ be a path on $N$ that $\gamma(t)$ lies in the interior of $N$ except $\gamma(0)=x$. Then the parallel transport of sections along $\gamma$ gives rise to ordinary differential equations in $t$, whose coefficients are uniformly bounded by assumption. In particular, since $\varphi$ is parallel, it satisfies the differential equations. Therefore $\varphi$ restricted on $\gamma$, as a function with value in $\R^{2^n}$, is Lipschitz continuous for $t>0$. Hence $\varphi$ along $\gamma$ extends continuously to $\gamma(0)=x$, which we will call $\varphi(x)$. The value $\varphi(x)$ is independent of the choice of $\gamma$.

Since $(M, g)$ is  a codimension zero flat submanifold of $\R^n$, there exist $n$ parallel sections  of $TM$, denoted by $\{v_1,v_2,\ldots,v_n \}$,  such that they form an orthonormal basis of $T_xM$ at every point $x\in M$. We also use the same notation to denote the standard basis in $\R^n$. Let $\Lambda$ be the collection of all subsets of $\{1,2,\ldots,n\}$. For $\lambda\in \Lambda$, we define
$$\omega_\lambda=\wedge_{i\in \lambda}v_i\in \Bigwedge^\ast TM$$
Note that $\{\omega_\lambda \}_{\lambda\in\Lambda}$ are parallel sections of $\Bigwedge^\ast TM$ such that they form an orthonormal basis of $\Bigwedge^\ast T_xM$ at every point $x\in M$.

With the section $\varphi$ of $S_N\otimes f^\ast S_M$ from above, we define
$$\varphi_\lambda=(1\otimes c(\omega_\lambda))\varphi.$$
Since $\omega_\lambda$ is parallel (with respect to the connection $\nabla$), we see that $\varphi_\lambda$ is parallel (with respect to the connection $\overbar \nabla \otimes 1 + 1\otimes \nabla$). Note that $\overbar\nabla\otimes 1+1\otimes\nabla$ is a Hermitian connection that preserves the inner product on $S_N\otimes f^\ast S_M$. Therefore, for any pair of elements $\lambda, \mu\in \Lambda$, the function  $\langle\varphi_\lambda(x),\varphi_\mu(x)\rangle$ (as $x$ varies over $N$) is a constant function. In particular, we may assume without loss of generality  that $|\varphi_\lambda(x)|=1$ for all $x\in N$ and all $\lambda\in \Lambda$.

\begin{claim*}
	The parallel sections $\{\varphi_\lambda\}_{\lambda\in\Lambda}$ are mutually orthogonal.
\end{claim*}

Note that $\dim(S_N\otimes f^\ast S_M)=2^n=|\Lambda|$, where $|\Lambda|$ is the cardinality of the set $\Lambda$.  Thus if the claim holds, then the curvature form of $S_N\otimes f^\ast S_M$ vanishes. Since $M$ is flat, the curvature form of $S_N\otimes f^\ast S_M$ is equal to the curvature form $R^{S_N}$ of $S_N$. By \cite[Theorem 2.7]{spinorialapproach}, we have
$$R^{S_N}_{X,Y}\sigma=\frac 1 2 R^{\overbar g}_{X,Y}\cdot \sigma,~ \textup{ for all } \sigma\in\Gamma(S_N) \textup{ and } X,Y\in\Gamma(TN),$$
where $R^{\overbar g}$ is the curvature form of the Levi-Civita connection on $TN$ with respect to $\overbar g$. It follows that $R^{\overbar g}=0$, that is, $\overbar g$ is flat.

Now we prove the claim. For each $\lambda \in \Lambda$ and $x\in M$, we denote by $V_\lambda$ the subspace in $T_xM \cong \mathbb R^n$ spanned by $\{v_i \}_{i\in\lambda}$.

Let $\lambda$ and $\mu$ be two distinct members of $\Lambda$.  Without loss of generality, we assume there exists $1\leq k \leq n$ such that $k\in \mu$ and $k\notin \lambda$. Equivalently, we have $v_k\in V_\lambda^\perp\cap V_\mu$. Consider the linear function $L(y) = \langle y,v_k\rangle$ on $M$ (viewed as a subspace of $\mathbb R^n$), which attains its minimum at some point $x\in M$, since $M$ is compact. Note that  $f\colon \partial N \to \partial M$ is surjective, since $\deg(f)\ne 0$. In particular, there is a point $\overbar x\in \partial N$ such that $f(\overbar x)=x$. Now there are two cases to consider. 

\textbf{Case I.} If $x$ is in the interior of a codimension one face $F_i$, then the unit inner normal vector $u$ of $F_i$ at $x$ is equal to $v_k$.  In this case, let $\overbar u$ be the unit inner normal vector at $\overbar x$ of the corresponding face $\overbar F_i$ of $N$.  Recall that the section $\varphi$ satisfies the following boundary condition at $\overbar x$:
$$(\overbar \epsilon\otimes\epsilon)(\overbar c(\overbar u)\otimes c(u))\varphi(\overbar x)=-\varphi(\overbar x).$$
Note that for a vector $v\in\R^n$, we have
$$c(\omega_\lambda)c(v)=\begin{cases}
	(-1)^{|\lambda|}c(v) c(\omega_\lambda),& \textup{ if } v\in V_\lambda^\perp, \vspace{.3cm}  \\
	(-1)^{|\lambda|-1}c(v) c(\omega_\lambda),& \textup{ if } v \in V_\lambda,
\end{cases}$$
where $|\lambda|$ is the cardinality of the set $\lambda$. Therefore $\varphi_\lambda$ and $\varphi_\mu$ satisfy the following equations at $\overbar x$: 
$$(\overbar \epsilon\otimes\epsilon)(\overbar c(\overbar u)\otimes c(u))\varphi_\lambda(\overbar x)=-\varphi_\lambda(\overbar x), $$
$$(\overbar \epsilon\otimes\epsilon)(\overbar c(\overbar u)\otimes c(u))\varphi_\mu(\overbar x)=\varphi_\mu(\overbar x).$$
It follows that  $\langle\varphi_\lambda,\varphi_\mu\rangle$ vanishes at $\overbar x$, hence everywhere on $N$.

\textbf{Case II.} Suppose $x$ lies in the interior of the intersection of  $m$ codimension one faces. Let $u_1,\ldots,u_m$ be the set of unit inner normal vectors of these $m$ codimension one faces. Then in this case, the vector $v_k$ from above lies in the linear span of $u_1,\ldots,u_m$ in $T_xM$. Since $\deg(f) \neq 0$, it follows from the definition of corner maps that there exists $\overbar x \in f^{-1}(x)$ such that $\overbar x$  lies in the intersection of $m$ codimension one faces of $N$. Indeed, $\deg(f) \neq 0$ implies that the map $f|_{\partial}\colon \partial N \to \partial M$ has nonzero degree. Since $f$ maps codimension $k$ faces of $N$ to codimension $k$ faces of $M$ for any $1\leq k \leq n$, it follows by induction that for each codimension $m$ face $F_\theta$ of $M$, there exists a codimension $m$ face $\overbar F_\theta$ of $N$ such that $f$ maps $\overbar F_\theta$ to $F_\theta$ with nonzero degree. This  in particular implies that $f$ maps $\overbar F_\theta$ surjectively onto $F_\theta$. To summarize, we see that there exists $\overbar x \in f^{-1}(x)$ such that $\overbar x$  lies in the intersection of $m$ codimension one faces of $N$.  We denote the corresponding unit inner normal vectors of these faces at $\overbar x$ by $\overbar u_1,\ldots,\overbar u_m$. Since $v_k$ lies in the linear span of $u_1,\ldots,u_m$, we have
$$v_k=\sum_{i=1}^m a_iu_i$$
for some numbers $a_1,\ldots,a_m\in \mathbb R$. Accordingly, we define
$$\overbar u\coloneqq \sum_{i=1}^m a_i\overbar u_i.$$
At the point $\overbar x$, the section $\varphi$ satisfies multiple boundary conditions, that is,
$$(\overbar \epsilon\otimes\epsilon)(\overbar c(\overbar u_i)\otimes c(u_i))\varphi(\overbar x)=-\varphi(\overbar x),~\forall i=1,\ldots m.$$
Equivalently, we have
$$(\overbar \epsilon\overbar c(\overbar u_i)\otimes 1)\varphi(\overbar x)=-(1\otimes \epsilon c(u_i))\varphi(\overbar x),~\forall i=1,\ldots m.$$
Since $(\overbar \epsilon\overbar c(\overbar u_i)\otimes 1)$ commutes with $1\otimes c(\omega_\lambda)$ and $1\otimes c(\omega_\mu)$, we have
$$(\overbar \epsilon\overbar c(\overbar u)\otimes 1)\varphi_\lambda(x)=-(1\otimes \epsilon c(v_k))\varphi_\lambda(\overbar x),$$
and
$$(\overbar \epsilon\overbar c(\overbar u)\otimes 1)\varphi_\mu(x)=(1\otimes \epsilon c(v_k))\varphi_\mu(\overbar x).$$
Note that  $\overbar c(\overbar u)(-\overbar c(\overbar u)^*)=\overbar c(\overbar u)^2=|\overbar u|^2_{\overbar g}$ and similarly $ c(v_k)(- c(v_k)^*) = |v_k|^2_{g} = 1$. It follows that 
$$|\overbar u|^2_{\overbar g}\langle\varphi_\lambda(\overbar x),\varphi_\mu(\overbar x)\rangle=-\langle\varphi_\lambda(\overbar x),\varphi_\mu(\overbar x)\rangle.$$
Since $|\overbar u|^2_{\overbar g}\geq 0$, this implies that  $\langle\varphi_\lambda(\overbar x),\varphi_\mu(\overbar x)\rangle=0$, hence
$\langle\varphi_\lambda,\varphi_\mu\rangle = 0$ everywhere on $N$. This proves the claim, hence shows that $(N, \overbar g)$ is flat. 

Now we shall prove part $(i)$, that is, the following claim. 
\begin{claim*}
	If a face $F_i$ of $M$ is flat, then the corresponding face $\overbar F_i$ in $N$ is also flat.
\end{claim*}

 Let $\{s_\alpha\}_{1\leq \alpha \leq 2^{n/2}}$ be a set of parallel sections of $f^\ast S_M$ such that they form an orthonormal basis of $(f^\ast S_M)_x$ for any point $x\in N$. Hence we can write
$$\varphi=\sum_{\alpha}\overbar s_\alpha\otimes s_\alpha,$$
where $\overbar s_\alpha$ are sections of $S_N$. Since $\varphi$ is parallel, each $\overbar s_\alpha$ is parallel with respect to $\overbar \nabla$.

From the above, we see that 
$$\{\varphi_\lambda=\sum_\alpha \overbar s_\alpha\otimes \omega_\lambda s_\alpha\}_{\lambda\in \Lambda}$$
 forms a basis of $(S_N\otimes f^\ast S_M)_{\overbar x}$ at every point $\overbar x\in N$. It follows that  
\[ \{\overbar s_\alpha\}_{1\leq \alpha \leq 2^{n/2}} \] forms a basis of $(S_N)_{\overbar x}$ at every $\overbar x\in N$. That is, $\{\overbar s_\alpha \}$ is linearly independent at any point in $N$.

At the face $\overbar F_i$ of $N$, the section $\varphi$ satisfies the boundary condition
$$(\overbar \epsilon\overbar c(\overbar e_n)\otimes 1)\varphi=-(1\otimes\epsilon c(e_n))\varphi,$$
where $\overbar e_n$ (resp. $e_n$) is the inner unit normal vector field of $\overbar F_i$ (resp.  the corresponding codimension one face $F_i$ in $M$). Therefore we have 
\begin{equation*}
	\sum_\alpha \overbar \epsilon\overbar c(\overbar e_n)\overbar s_\alpha\otimes s_\alpha=\sum_\alpha \overbar s_\alpha\otimes \epsilon c(e_n) s_\alpha.
\end{equation*}
 Note that $e_n$ is parallel. Let $X$ be an arbitrary tangent vector field along $\overbar F_i$. By applying $(\overbar\nabla_X\otimes 1+1\otimes\nabla_X)$ to both sides of the above equality, we obtain
 $$\sum_\alpha \overbar \epsilon\overbar c(\overbar \nabla_X\overbar e_n)\overbar s_\alpha\otimes s_\alpha=0.$$
 Therefore $c(\overbar \nabla_X\overbar e_n)\overbar s_\alpha=0$ for all $\alpha \in \{1,2, \ldots,2^{n/2}\}$. It follows that $\overbar \nabla_X\overbar e_n=0$ for all tangent vector fields $X$ along $\overbar F_i$, that is, the second fundamental form of $\overbar F_i$ vanishes. As we have shown that $N$ is flat, this implies that $\overbar F_i$ is also flat. 
 
 Now let us prove part $(ii)$.  By assumption,  $\overbar x$ is a point in the intersection of $m$ codimension one faces of $N$ whose unit inner normal vectors at $\overbar x$ are denoted by $\overbar \nu_1,\ldots,\overbar \nu_m$. Let $\nu_1, \ldots, \nu_m$ be the unit inner normal vectors at $f(\overbar x)$ of the corresponding faces of $M$. Again, let $\varphi$ be  the parallel section of $S_N\otimes f^\ast S_M$ from above. By the above discussion, we have
 $$(\overbar \epsilon\overbar c(\overbar \nu_j)\otimes 1)\varphi(\overbar x)=-(1\otimes \epsilon c(\nu_j))\varphi(\overbar x),~\forall j=1,\ldots m.$$
 For $a=(a_1,a_2,\ldots,a_m)\in\R^m$, we define
 $$\nu_a\coloneqq \sum_{j=1}^m a_j\nu_j \textup{ and } \overbar \nu_a\coloneqq \sum_{j=1}^m a_j\overbar \nu_j.$$
 Clearly, we have 
 $$(\overbar \epsilon\overbar c(\overbar \nu_a)\otimes 1)\varphi(\overbar x)=-(1\otimes \epsilon c(\nu_a))\varphi(\overbar x).$$
 By taking vector norms of both sides, we obtain
 $$|\overbar \nu_a|_{\overbar g}^2\cdot |\varphi(\overbar x)|^2=|\nu_a|_g^2\cdot |\varphi(\overbar x)|^2$$
 since $c(\nu_a)(-c(\nu_a)^*)=c(\nu_a)^2=|\nu_a|_{g}^2$ and $\overbar c(\overbar \nu_a)(-\overbar c(\overbar \nu_a)^*)=\overbar c(\overbar \nu_a)^2=|\overbar \nu_a|_{\overbar g}^2$. It follows that 
 $|\overbar \nu_a|_{\overbar g}=|\nu_a|_g$, since $|\varphi(\overbar  x)|\ne 0$. 
 
 Consider the two symmetric quadratic forms
 $$\overbar Q,Q\colon \R^m\times\R^m\to\R$$
 defined by
 $$\overbar Q(a,b)\coloneqq \langle \overbar \nu_a,\overbar \nu_b\rangle_{\overbar g} \textup{ and } Q(a,b)\coloneqq \langle \nu_a,\nu_b\rangle_{g}.$$
 The above discussion shows that 
 $$Q(a,a)=\overbar Q(a,a),~\forall a\in\R^m.$$
 By the polarization identity, we have
 $$Q(a,b)=\frac 1 4(Q(a+b,a+b)-Q(a-b,a-b)),$$
 $$\overbar Q(a,b)=\frac 1 4(\overbar Q(a+b,a+b)-\overbar Q(a-b,a-b)).$$
 Hence $Q$ and $\overbar Q$ are identical. In particular, we have 
 $$\langle \nu_i,\nu_j\rangle_g=\langle \overbar \nu_i,\overbar \nu_j\rangle_{\overbar g},~\forall i,j=1,\ldots,m.$$
 
 Now let us prove part $(iii)$. Note that the proof of part $(ii)$ essentially relies on the fact that the section $\varphi$ satisfies multiple boundary conditions at a singular point of the boundary. We shall prove part $(iii)$ in a similar way, by  applying the fact that $\varphi$ satisfies the boundary conditions at any two different points $\overbar x_1,\overbar x_2\in \partial N$. More precisely,  suppose that $\overbar x_1,\overbar x_2\in \partial N$ and $f(\overbar x_1),f(\overbar x_2)\in \partial M$ are in the interior of codimension one faces. Denote their corresponding unit inner normal vectors  by $\nu_{\overbar x_1}$, $\nu_{\overbar x_2}$, $\nu_{f(\overbar x_1)}$ and $\nu_{f(\overbar x_2)}$.  Let $\gamma\colon [0,1]\to N$ be a piecewise smooth path from $\overbar x_1$ to $\overbar x_2$ in $N$ and $f\gamma$ be its image in $M$. Let $\widetilde \nu$ be the parallel transport of $\nu_{\overbar x_1}$ along $\gamma$, that is, $\widetilde\nu(0)=\nu_{\overbar x_1}$ and $\ncon_{\dot\gamma} \widetilde\nu=0$, where  $\ncon$ is the Levi--Civita connection on $N$. We define  $\widetilde\nu_{\overbar x_1}\coloneqq \widetilde\nu(1)$. 
 Similarly, let $\widetilde \nu_{f(\overbar x_1)}$ be the parallel transport of $ \nu_{f(\overbar x_1)}$ from $f(\overbar x_1)$ to $f(\overbar x_2)$ via the path $f\gamma$. 
 
 Since $\varphi$ satisfies the boundary condition at $\overbar x_1$, we have
 $$\big(\overbar\epsilon\overbar c(\nu_{\overbar x_1})\otimes 1\big)\varphi(\overbar x_1)= -\big(1\otimes \epsilon c(\nu_{f(\overbar x_1)})\big)\varphi(\overbar x_1).$$
 The parallel transport of both sides with respect to $\nabla$ on $S_N\otimes f^*S_M$ gives the following equation
 $$\big(\overbar\epsilon\overbar c(\widetilde \nu_{\overbar x_1})\otimes 1\big)\varphi(\overbar x_2)= -\big(1\otimes \epsilon c(\widetilde \nu_{f(\overbar x_1)})\big)\varphi(\overbar x_2)$$
 at $\overbar x_2$, 
since $\varphi$ is parallel. Moreover, 
$\varphi$ also satisfies the boundary condition at $\overbar x_2$
  $$\big(\overbar\epsilon\overbar c(\nu_{\overbar x_2})\otimes 1\big)\varphi(\overbar x_2)= -\big(1\otimes \epsilon c(\nu_{f(\overbar x_2)})\big)\varphi(\overbar x_2).$$
  Now we apply the same argument  in the proof of part $(ii)$ by considering linear combinations of the above two equations at $\overbar x_2$. It follows that 
 $$\langle \widetilde{\nu}_{\overbar x_1},\nu_{\overbar x_2}\rangle=\langle \widetilde \nu_{f(\overbar x_1)},\nu_{f(\overbar x_2)}\rangle.$$ 
 This finishes the proof.	

\end{proof}

Observe that in the proof of Theorem \ref{thm:strongflat} above, if there is a point $x$ of $M$ such that the inner normal vectors (of codimension one faces) at $x$ span the whole tangent space $T_xM$, then we can deduce the flatness of $(N, \overbar g)$ by only using these inner normal vectors. This leads us to  the following notion of generalized vertices in manifolds with polyhedral boundary. 
\begin{definition}
	Let $M$ be an $n$-dimensional flat manifold with polyhedral boundary. For a given $x$ in $\partial M$, let $U$ be a neighborhood of $x$ in $\partial M$.  We denote by  $\mathcal L_U$ the set of vectors in $T_xM$ consisting of parallel transports of inner normal vectors of $\partial M $ at points $y\in U$.  We say the point $x$ is a \emph{generalized vertex} if the linear span of $\mathcal L_U$ is $T_xM$.
\end{definition}
For example, a usual vertex of any Euclidean polyhedron is clearly a generalized vertex in the sense of the above definition.

The following theorem generalizes Theorem \ref{thm:strongflat} to a class of flat manifolds that admit generalized vertices.  
 
\begin{theorem}\label{thm:flatWithVertex}
		Suppose $M$ is a compact oriented $n$-dimensional flat manifold with polyhedral boundary such that 
	\begin{enumerate}[label=$(\alph*)$]
		\item $M$ has nonzero Euler characteristic,
		\item each of its codimension one face is convex, that is, its second fundamental form is non-negative, 
		\item all of its dihedral angles are $< \pi$,
		\item $M$ admits a generalized vertex. 
	\end{enumerate}   
	Let $(N, \overbar{g})$ be a compact oriented $n$-dimensional spin manifold with polyhedral boundary. If $f\colon (N, \overbar{g}) \to (M, g)$  is conically smooth map such that 
	\begin{enumerate}[label=$(\arabic*)$]
		\item $f$ is distance-non-increasing on $\partial N$, 
		\item the degree of $f$ is nonzero,	
		\item 
		$\Sc(\overbar{g})_x \geq \Sc(g)_{f(x)} = 0$ for all $x\in N$, 
		\item $H_{\overbar{g}}(\overbar{F}_i)_y \geq  H_{g}(F_i)_{f(y)}$ for all $y$ in each codimension one face $\overbar{F}_i$ of $N$, 
		\item $\theta_{ij}(\overbar{g})_z\leq \theta_{ij}(g)_{f(z)}$ for all $\overbar F_i, \overbar F_j$ and all $z \in \overbar F_i\cap \overbar{F}_j$,   
	\end{enumerate}
	then 
\begin{enumerate}
	 \item 
	$\Sc(\overbar{g})_x = \Sc(g)_{f(x)} = 0$ for all $x\in N$, 
	\item $H_{\overbar{g}}(\overbar{F}_i)_y =  H_{g}(F_i)_{f(y)}$ for all $y$ in each codimension one face $\overbar{F}_i$ of $N$, 
	\item $\theta_{ij}(\overbar{g})_z =  \theta_{ij}(g)_{f(z)}$ for all $\overbar F_i, \overbar F_j$ and all $z \in \overbar F_i\cap \overbar{F}_j$, 
\end{enumerate}
and $(N, \overbar g)$ is also flat. 
	Furthermore, the following hold. 
	\begin{enumerate}[label=$(\roman*)$]
		\item If a codimension one face $F_i$ of $M$ is flat, then the corresponding face $\overbar F_i$ of $N$ is also flat.
		\item Suppose in addition that the preimage $f^{-1}(F_i)$ of each codimension one face $F_i$ of $M$ is equal to a codimension one face of $N$. If all codimension one faces of $M$ are flat, then  the manifold $(N, \overbar g)$ at each point $\overbar x\in N$ is locally isometric to $(M, g)$ at $f(\overbar x)\in M$. 
	\end{enumerate}
\end{theorem}
\begin{proof}
	Without loss of generality, we assume that both $N$ and $M$ are even dimensional. By assumption, we pick $n$ points $x_1,x_2,\ldots,x_n$ near the vertex $x\in\partial M$ such that the unit inner normal vectors $v_1,v_2,\ldots,v_n$ parallel transports to a spanning set of $T_xM$. Note that the parallel transport is independent of path near $x$ as $M$ is flat.
	
	Let $\overbar x$ be a preimage of $x$ in $\partial N$ and $\overbar x_i$ a preimage of $x_i$ such that $\overbar x_i$ lies in the interior of a face of $N$. Let $\overbar v_i$ be the unit inner normal vector of $\overbar x_i$. Fix a choice of $n$ paths from $\overbar x_i$ to $\overbar x$ and denote by $\widetilde v_i$ the parallel transport of $\overbar v_i$ to $\overbar x$. So far the vector $\widetilde v_i$ in $T_{\overbar x}N$ may depend on the choice of the paths.
	
	By the same proof of part (ii) of Theorem \ref{thm:strongflat}, we have that
	$$\langle\widetilde v_i,\widetilde v_j\rangle=\langle v_i,v_j\rangle,~\forall i,j,$$
	and	the parallel solution $\varphi$ in $S_N\otimes f^*S_M$ satisfies
	$$(\overbar\epsilon\overbar c(\widetilde v_i)\otimes 1)\varphi(\overbar x)=-(1\otimes c(v_i))\varphi(\overbar x).$$
	
	Since $v_1,\ldots,v_n$ spans $T_xM$, there exists an $n\times n$ matrix $(a_{ij})$ such that the vectors
	$$w_i\coloneqq \sum_{j=1}^n a_{ij}v_j$$
	form an orthonormal basis for $T_xM$. Therefore,
	$$\widetilde w_i\coloneqq\sum_{j=1}^n a_{ij}\widetilde v_j$$
	also form an orthonormal basis for $T_{\overbar x}N$.
	
	We first assume that $M$ simply connected. In this case, the vectors $\{w_i\}$ extends to global parallel sections of $TM$ by parallel transport, which we still denote by $\{w_i\}$. Now we are able to apply the same proof of Theorem \ref{thm:strongflat}. Let $\Lambda$ be the collection of all subsets of $\{1,2,\ldots,n\}$. For $\lambda\in \Lambda$, we define
$$\omega_\lambda=\wedge_{i\in \lambda}w_i,$$
which are parallel sections of $\Bigwedge^\ast TM$ and form an orthonormal basis of $\Bigwedge^\ast T_z M$ at every point $z\in  M$.  We define a collection of non-zero parallel sections
	$$\varphi_\lambda=(1\otimes \omega_\lambda)\varphi,~\forall\lambda\in\Lambda.$$
	Since $\varphi_\lambda$ is parallel, we see that
	$\langle \varphi_\lambda (x),\varphi_\mu(x)\rangle$ (as $x$ varies over $\widetilde N$) is a constant function  $\widetilde N$,  for any $\lambda,\mu\in \Lambda.$
	
	Note that $\varphi$ at $\overbar x$ satisfies
	$$(\overbar\epsilon\overbar c(\widetilde w_i)\otimes 1)\varphi(\overbar x)=-(1\otimes c(w_i))\varphi(\overbar x),~ \forall i=1,\ldots,n$$
	from the argument above. From the same argument in the proof of Theorem \ref{thm:strongflat}, we obtain that 
	$$\langle \varphi_\lambda(\overbar x),\varphi_\mu(\overbar x)\rangle=0,~\forall \lambda,\mu\in\Lambda,~\lambda\ne\mu.$$
	Therefore $\{\varphi_\lambda \}_{\lambda\in\Lambda}$ forms a parallel basis of $S_{\widetilde N}\otimes \widetilde f^*S_{\widetilde M}$. It follows that $N$ is flat.
	
	In general, if $M$ is not simply connected, we consider its universal cover $\widetilde M$. The map $f$ lifts to $\widetilde f\colon \widetilde N\to\widetilde M$, where $\widetilde N$ is the pullback cover of $N$. Now the bundle $S_N\otimes f^*S_M$ lifts to $\widetilde N$, the point $x\in\partial M$ lifts to a vertex of $\widetilde M$, and $\varphi$ lifts to a parallel section. The same proof as above shows that $\widetilde N$ is flat, hence $N$ is flat.
	
	Now suppose a codimension one face $F_i$ of $M$ is flat. The same computation from the proof of part $(i)$ of Theorem \ref{thm:strongflat} shows  the corresponding face $\overbar F_i$ of $N$ is also flat. This proves part $(i)$. 
	
	Now let us prove part $(ii)$. Since now all codimension one faces of $M$ are assumed to be flat,  it follows from part $(i)$ that all codimension one faces of $N$ are flat. In particular, the local geometry of $(N, \overbar g)$ near each codimension $k$ face $\overbar F^k$ is completely determined by the unit inner normal vectors $\{\overbar v_1, \cdots, \overbar v_\ell\}$ of the codimension one faces that contain $\overbar F^k$. The same remark also holds for $(M, g)$. By assumption, the preimage $f^{-1}(F_i)$ of each codimension one face $F_i$ of $M$ is equal to a codimension one face of $N$. It follows that $f$ maps each interior point $\overbar x$ of $\overbar F^k$ to an interior point of a codimension $k$ face of $M$. The same argument for part $(ii)$ of Theorem \ref{thm:strongflat} shows that  the relative positions of the above unit inner normal vectors $\{\overbar v_1, \cdots, \overbar v_\ell\}$ coincide with the relative positions of the unit inner normal vectors of the corresponding faces of $(M, g)$.   It follows that  $(N,\overbar g)$ at $\overbar x$ is  locally isometric to $(M, g)$ at $f(\overbar x)$. This  proves part $(ii)$, hence finishes the proof of the theorem. 
\end{proof}

It is clear that Theorem \ref{thm:polyhedra} is an immediate consequence of Theorem \ref{thm:flatWithVertex}. Now let us deduce  Theorem \ref{thm:stokereuclidean} from either Theorem \ref{thm:strongflat} or Theorem  \ref{thm:flatWithVertex}. 
\begin{proof}[Proof of Theorem \ref{thm:stokereuclidean}]
	Let $P_1$ and $P_2$ be two convex polyhedra in $\R^n$. By taking direct product with the unit interval $[0, 1]$ if necessary, we assume without loss of generality that $n$ is even.

	Since the combinatorial types of $P_1$ and $P_2$ are the same, there is a homeomorphism $f\colon P_1\to P_2$ that preserves the combinatorial structures and matches the dihedral angles. The map $f$ may not be smooth on the entire $N$, but it can be chosen to be a polytope map. For example, we may define $f$ first away from codimension $3$ faces, and then inductively extend $f$ radially to faces of higher codimensions. We identify the spinor bundle $S_{P_1}\otimes f^*S_{P_2}$ with the bundle of differential forms $\Bigwedge^\ast T\R^n$ over $P_1$.

	By part (ii) of Theorem \ref{thm:strongflat}, at any given vertex $x\in P_1$,  the relative positions of the unit inner normal vectors of codimension one faces coincide with the relative positions of the unit inner normal vectors of the corresponding faces of $P_2$. It follows that the corresponding face angles of $P_1$ and $P_2$ coincide. This finishes the proof.
\end{proof}

\section{Rigidity of flat domains with smooth boundary}\label{sec:smooth}
In this section, we investigate rigidity results for flat domains with smooth boundary. More precisely,  we prove a scalar-mean rigidity theorem for strictly convex domains with smooth boundary in Euclidean spaces (Theorem \ref{thm:balliso}). 

As a preparation, we first prove the following  proposition.

\begin{proposition}\label{prop:strictlyConvex}
	Let $(M, g)$ and $(N, \overbar g)$ be two compact $n$ dimensional manifolds with smooth boundary, and $f\colon N\to M$ a spin map. Suppose \begin{enumerate}[label=$(\arabic*)$]
		\item $M$ has nonzero Euler characteristic,
		\item the curvature operator of $(M,g)$ is non-negative,
		\item the second fundamental form of $\partial M$ is strictly positive,  
		\item $f$ is area-non-increasing on $N$  and $f$ is distance-non-increasing on $\partial N$, 
		\item the degree of $f$ is nonzero,	
		\item 
		$\Sc(\overbar{g})_x \geq \Sc(g)_{f(x)} = 0$ for all $x\in N$, 
		\item $H_{\overbar{g}}(\overbar{F}_i)_y \geq  H_{g}(F_i)_{f(y)}$ for all $y\in \partial N$. 
	\end{enumerate}   
	Then $f\colon \partial N\to\partial M$ is a local isometry.
\end{proposition}

\begin{proof}
	The odd dimensional case can easily be reduced to the even dimensional case by considering $N\times [0,1]$ and $M\times[0,1]$ with the obvious product metrics. Therefore, we can assume without loss of generality  that $N$ and $M$ are even-dimensional. 
	
	Suppose $f(y) =x \in \partial M$ for a given point $y\in \partial N$. We diagonalize the metric at $y$ and $x$,  and choose an orthonormal basis $\{\overbar e_j \}_{1\leq j \leq n-1}$ of $T_y(\partial N)$ and an orthonormal basis $\{e_j \}_{1\leq j\leq n-1}$ of $T_x(\partial M)$ such that
	$$f_*\overbar e_j=\alpha_j e_j$$
	for some $\alpha_j\in [0,1]$, since $f$ is distance-non-increasing on $\partial N$. We denote by $\overbar e_n$ (resp. $e_n$)  the unit inner normal vector of $\partial N$ at $y$ (resp. of $\partial M$ at $x$).
	
	Let $A$ be the second fundamental form of $\partial M$. Let $\overbar c_\partial$ and $c_\partial$ be the boundary Clifford actions, that is, $\overbar c_\partial(\overbar e_j)=\overbar c_\partial(\overbar e_n)\overbar c_\partial(\overbar e_j)$ and $c_\partial(e_j)=c(e_n)c(e_j)$.

	Since $A$ is strictly positive, there is an invertible endomorphism $L\in\mathrm{End}(T_x(\partial M))$ such that $A=L^2$, that is,
	$$A(e_i,e_j)=\langle Le_i,Le_j\rangle.$$
	Let us write 
	$$Le_i=\sum_{1\leq j\leq n-1} L_{ij}e_j$$
	and
	$$\overbar L e_i=\sum_{1\leq j\leq n-1}\langle Le_i,f_*\overbar e_j\rangle\overbar e_j=\sum_{1\leq j\leq n-1} L_{ij}\alpha_j\overbar e_j .$$
	We have the following inequality (cf. \cite[Lemma 2.3]{Wang:2021tq})
	\begin{align*}
		&\sum_{i,j}A(f_*\overbar e_i, e_j)\overbar c_\partial(\overbar e_i)\otimes c_\partial(e_j)	=\sum_{i,j,k}\langle L(f_*\overbar e_i),e_k\rangle_M 
		\langle L(e_j),e_k\rangle_M  \overbar c_\partial(\overbar e_i)\otimes c_\partial(e_j)\\
		=&\sum_{k}\overbar c_\partial(\overbar Le_k)\otimes c_\partial(L e_k)\\
		=&-\frac 1 2\sum_k\Big(\overbar c_\partial(\overbar Le_k)^2\otimes 1+1\otimes c_\partial(L e_k)^2-\big(\overbar c_\partial(\overbar Le_k)\otimes 1+1\otimes c_\partial(L e_k)\big)^2\Big)\\
		\leq&-\frac 1 2\sum_k \overbar c_\partial(\overbar Le_k)^2\otimes 1 - \frac{1}{2}\sum_k 1\otimes   c_\partial(Le_k)^2,
	\end{align*}
	where the last inequality  follows from the fact  the element
	\[  \big(\overbar c_\partial(\overbar Le_k)\otimes 1+1\otimes  c_\partial(L e_k)\big) \] is skew-symmetric, hence its square is non-positive.  
	Note that
	\begin{align*}
		\sum_k c_\partial(Le_k)^2=&\sum_{k,j}L_{kj}L_{kj}c_\partial(e_j)^2+\sum_k\sum_{i\ne j}L_{ki}L_{kj}c_\partial(e_i)c_\partial(e_j)\\
		=&-\sum_{j}A_{jj}+\sum_{i\ne j}A_{ij}c_\partial(e_i)c_\partial(e_j)=-H_g,
	\end{align*}
	and
	\begin{align*}
		\sum_k \overbar c_\partial(\overbar Le_k)^2=&\sum_{k,j}\alpha^2_jL_{kj}L_{kj}\overbar c_\partial(\overbar e_j)^2+\sum_k\sum_{i\ne j}\alpha_i\alpha_jL_{ki}L_{kj}c_\partial(e_i)c_\partial(e_j)\\
		=&-\sum_{j}A_{jj}\alpha^2_j+\sum_{i\ne j}\alpha_i\alpha_jA_{ij}c_\partial(e_i)c_\partial(e_j)\\
		=&-\sum_j A_{jj}\alpha^2_j\geq-\sum_j A_{jj}=-H_g.
	\end{align*}
	To summarize, we have 
	$$\sum_{i,j}A(f_*\overbar e_i,e_j)\overbar c_\partial(\overbar e_i)\otimes c_\partial(e_j)\leq f^\ast(H_g)$$
	at $y\in \partial N$. 
	
	Let  $\varphi$ be a parallel section of  $S_N\otimes f^*S_M$ as in the proof of Theorem \ref{thm:strongflat}. The fact that $\varphi$ satisfies $D\varphi =0$ together with the assumptions on comparisons of scalar curvature and mean curvature show that the above inequality 
	$$\sum_{i,j}A(f_*\overbar e_i,e_j)\overbar c_\partial(\overbar e_i)\otimes c_\partial(e_j)\leq f^*(H_g)$$
	becomes equality at $\varphi$ (cf. \cite[Lemma 2.3 \& Proposition 2.6]{Wang:2021tq}), that is, 
	\begin{equation}\label{eq:meanCurvatureEquality}
		\big(\sum_{i,j}A(f_*\overbar e_i,e_j)\overbar c_\partial(\overbar e_i)\otimes c_\partial(e_j)\big)\varphi=f^*(H(g))\cdot\varphi.
	\end{equation}
	It follows that   
	$$\sum_j A_{jj}\alpha^2_j=\sum_j A_{jj}.$$
	Since $A$ is strictly positive, $A_{jj}>0$ for all $1\leq j\leq n-1$. Therefore
	$\alpha_j=1$ for all $1\leq j\leq n-1$, which shows that  $f_*\colon T_y(\partial N)\to T_x(\partial M)$ is an isometry. It follows that $f\colon \partial N \to \partial M$ is a local isometry. This finishes the proof. 
\end{proof}

Before we prove Theorem \ref{thm:balliso}, we first prove the following variant of Theorem  \ref{thm:balliso}. Note that the assumption on the map $f$ in  Theorem \ref{thm:ballAlmostIso} below is weaker than the corresponding assumption on $f$ in Theorem \ref{thm:balliso}. On the other hand, the conclusion of  Theorem \ref{thm:ballAlmostIso} below is also weaker than the conclusion of  Theorem \ref{thm:balliso}. 
\begin{theorem}\label{thm:ballAlmostIso}
	Let $(M,g)$ be a strictly convex domain  with smooth boundary in $\R^n$ \textup{($n\geq 2$)}. Let $(N,\overbar g)$ be an $n$ dimensional spin Riemannian manifold with boundary and $f\colon N\to M$ be a spin map. If
	\begin{enumerate}[label=$(\arabic*)$]
		\item 
		$\Sc(\overbar{g})_x \geq \Sc(g)_{f(x)} = 0$ for all $x\in N$, 
		\item $H_{\overbar{g}}(\partial N)_y \geq  H_{g}(\partial M)_{f(y)}$ for all $y\in \partial N$, 
		\item $\partial f \coloneqq f|_{\partial N} \colon \partial N\to \partial M$ is distance-non-increasing, 
		\item the degree of $f$ is nonzero,	
	\end{enumerate} 
	then $N$ is also a strictly convex domain in $\R^n$. Moreover, up to an affine isometry of $\R^n$,  $N$ coincides with $M$ and $\partial f\coloneqq f|_{\partial N} \colon \partial N\to \partial M$ becomes the identity map. 
\end{theorem}
\begin{proof}
		We assume that $n=\dim N=\dim M$ is even.	The odd dimensional case reduces to the even dimensional case by considering $f\times \id \colon N\times[0,1]\to M\times[0,1]$.
	
	In even dimensions, 	as a special case of Theorem \ref{thm:previous}, the vector bundle $S_N\otimes f^*S_M$  admits a non-zero parallel section $\varphi$ satisfying the boundary $B$, and we have
	\begin{enumerate}[label=$(\arabic*)$]
		\item 
		$\Sc(\overbar{g})_x = \Sc(g)_{f(x)} = 0$ for all $x\in N$, 
		\item $H_{\overbar{g}}(\partial N)_y =  H_{g}(\partial M)_{f(y)}$ for all $y\in \partial N$.
	\end{enumerate}  
	We mention that, in the even dimensional case, this also follows from \cite[Theorem 1.1]{Lottboundary}. Moreover, since $\partial M$ is strictly convex, it follows from Proposition \ref{prop:strictlyConvex} that  $f\colon \partial N\to \partial M$ is a local  isometry. 
	
	By Theorem \ref{thm:strongflat}, $(N,\overbar g)$ is also flat. We now show that $f$ preserves the second fundamental forms of $\partial N$ and $\partial M$. For any $x\in\partial N$, we choose $n$ points $\{x_1,\ldots,x_n\}$ near $x$ in $\partial N$ so that the set $\{v_1, \ldots, v_n\}$ is linearly independent in $\R^n$, where $v_i$ is the unit inner normal vector of $\partial M$  at $f(x_i)$. Such a set of points  always exits since $\partial M$ is strictly convex. Let $\overbar v_i$ be the inner normal vector of $\partial N$ at $x_i$. The parallel transport of each $\overbar v_i$ gives a parallel vector field near $x$, which we still denote by $\overbar v_i$. By $(iii)$ in Theorem \ref{thm:strongflat}, we have
	$$\langle \overbar  v_i,\overbar v_j\rangle=\langle  v_i, v_j\rangle,~\forall i,j.$$
	In particular, the set of vectors $\{\overbar v_1, \ldots, \overbar v_n\}$ is also linearly independent.
	
	Let $\overbar v$ (resp. $ v$)  be the unit inner normal vector field of $\partial N$ (resp. $\partial M$) near $x$ (resp. $f(x)$). By $(iii)$ in Theorem \ref{thm:strongflat} we have that $\langle \overbar v,\overbar v_i\rangle=\langle v, v_i\rangle$ for all  $1\leq i\leq n$. Therefore, $\overbar v$ and $ v$ can be  written as linear combinations of $\overbar v_i$'s and $ v_i$'s with the same coefficients. In other words, there are smooth functions $k_1,\ldots,k_n$ defined on a neighborhood of $x$ in $\partial M$ such that
	$$ v=\sum_{i=1}^n k_i v_i \text{ and }\overbar v=\sum_{i=1}^n (f^*k_i)\overbar v_i.$$
	Let $\overbar w$ be an arbitrary  vector field tangent to $\partial N$. Since $ v_i$ and $\overbar  v_i$ are parallel, we have
	$$\nabla_{f_*\overbar w}^M v=\sum_{i=1}^n f_*\overbar w(k_i)\cdot  v_i \textup{ and }\nabla^N_{\overbar w}\overbar v=\sum_{i=1}^n \overbar w(f^*k_i)\cdot \overbar v_i.$$
	Note that $\overbar w(f^*k_i)=f^*(f_*\overbar w(k_i))$ by the chain rule. Therefore for any vector fields $\overbar w, \overbar u$ tangent to  $\partial N$, we have
	$$\langle\nabla^N_{\overbar w}\overbar v,\nabla^N_{\overbar u}\overbar  v\rangle=\langle \nabla^M_{f_*\overbar w} v,\nabla^M_{f_*\overbar u} v\rangle.$$
	Let $\{\overbar w_1, \ldots, \overbar w_{n-1}\}$ be a local orthonormal basis of $T\partial N$ near $x$. As $f$ is a local  isometry from $\partial N$ to $\partial M$, $\{f_*\overbar w_1, \ldots, f_*\overbar w_{n-1}\}$ is also a local orthonormal basis of $T\partial M$ near $f(x)$. Let $A=(A_{jk})$ and $\overbar A=(\overbar A_{jk})$ be the second fundamental forms of $\partial M$ and $\partial N$, that is,  
	$$\nabla^N_{\overbar w_j}\overbar v=-\sum_{k=1}^{n-1}\overbar A_{jk}\overbar w_k,\text{ and }\nabla^N_{f_*\overbar w_j} v=-\sum_{k=1}^{n-1}A_{jk}f_*\overbar w_k.$$
	Since $M$ is strictly convex, we assume without loss of generality that $A$ is a diagonal matrix with positive diagonal entries. 
	
	By rewriting   $\langle\nabla^N_{\overbar w}\overbar v,\nabla^N_{\overbar u}\overbar  v\rangle=\langle \nabla^M_{f_*\overbar w} v,\nabla^M_{f_*\overbar u} v\rangle$ in terms of the above  matrix entries,  we obtain that 
	\[ \overbar A^{\, 2}=A^2. \] Since $\overbar A$ is symmetric, we have that
	$O=\overbar AA^{-1}$ is an orthogonal matrix. Note that
	$$\tr(\overbar A)=\tr(OA)=\sum_{j=1}^{n-1} O_{jj}A_{jj}\leq \sum_{j=1}^{n-1} A_{jj}=\tr(A),$$
	where the second equality is because  $A$ is diagonal, and the third inequality is because $|O_{jj}|\leq 1$, since $O$ is orthogonal. 
	Recall that the mean curvature of $N$ and $M$ are equal, that is, $\tr(\overbar A)=\tr(A)$. This implies that  $O_{jj} =1$ for each $1\leq j\leq n-1$. Therefore, $O$ is the identity matrix. It follows that $f$ preserves the second fundamental forms.

	Let $\widetilde N$ be the universal cover of $N$ with the lift metric, and $\pi\colon \widetilde N\to N$ the covering map.	
	Fix a point $\tilde x$ in $\widetilde N$ and an orthonormal frame at $\tilde x$. The parallel transport of this orthonormal frame at $\tilde x$ defines a set of global  orthonormal basis $\{\overbar e_1, \ldots, \overbar e_n\}$ of $T\widetilde N$, where each $\overbar e_i$ is parallel. For any $\tilde y\in \widetilde N$, choose a smooth path $\gamma$ connecting $\tilde x$ and $\tilde y$, and define 
	\[ \tilde y_i \coloneqq \int_{\gamma} \langle \dot \gamma, \overbar e_i\rangle  \]
	where $\dot \gamma$ is the tangent vector of $\gamma$. Since $\widetilde N$ is simply connected and flat, the above integral is independent of the choice  of $\gamma$ among all smooth curves connecting $\tilde x$ and $\tilde y$. These functions $\tilde y_i\colon \widetilde N\to\R$   together give rise to a map $p\colon \widetilde N\to \R^n$ that is locally isometric. 
	
	Our next step is to show that $p\colon \widetilde N \to p(\widetilde N)$ is a Riemannian covering map.  First we   show that $\partial N$ has only one connected component. Otherwise, fix a connected component $C$ of $\partial N$. The distance from $C$ to $\partial N-C$ is positive, and is attained by some geodesic $\gamma$ connecting $x\in C$ and $y\in C'$, where $C'$ is a connected component of $\partial N - C$. Since the length of $\gamma$ is the minimum among all curves connecting $C$ to $\partial N - C$, it follows that $\gamma$ is orthogonal to both $C$ and $C'$, and lies in the interior of $N$ except the two end points. Let $U$ be a small neighborhood of $\gamma$. 
	Since $N$ is flat, $U$ embeds isometrically into $\R^n$. Such an embedding maps $\gamma$  to a  line segment. Now since both $C$ and $C'$ are strictly convex, any line segment from $C$ to $C'$ inside $U$ parallel to $\gamma$ shorten the distance. This contradicts  the minimality of the chosen geodesic $\gamma$ and proves the claim.

	The exact same argument above also shows that $\partial \widetilde N$ has only one connected component.	Therefore $p(\partial\widetilde N)$ is connected, and $p\colon\partial\widetilde N\to p(\partial\widetilde N)\subset\R^n$ is a Riemannian submersion. For any $x\in \partial\widetilde N$, let $U_x$ be a small neighborhood near $x$ in $\partial\widetilde N$ such that $p|_{U_x}$ is a Riemannian embedding. Then $f\circ\pi\circ p^{-1}$ is an isometry from $p(U_x)$ to a small neighborhood $V_{\pi\circ f(x)}$ of $\pi\circ f(x)$ in $\partial M$, both of which are hypersurfaces in $\R^n$, and preserves the second fundamental form. In other words, the two hypersurfaces $p(U_x)$ and $V_{\pi\circ f(x)}$ in $\R^n$ have the same first and second fundamental forms. Therefore, it follows from the uniqueness of the solution to the partial differential equations describing the hypersurface that $p(U_x)$ and $V_{\pi\circ f(x)}$ are the same in $\R^n$ up to an orthogonal transform. Therefore, by continuously extend the argument to the whole $\partial\widetilde N$, we see that $p(\partial \widetilde N)$ and $\partial M$ in $\mathbb R^n$ are equal in $\R^n$ up to an orthogonal transform.  Moreover, $p$ restricted to $\partial\widetilde N$ is a Riemannian covering map.
	
	We claim that if $p(x)\in p(\partial\widetilde N)$, then $x\in\partial\widetilde N$. In other words, the map $p$ will never map an interior point of $\widetilde N$ to $p(\partial \widetilde N)$.   Assume to the contrary that there exists $x$ in the interior of $\widetilde N$ such that $p(x)\in p(\partial\widetilde N)$. The distance from $x$ to $\partial \widetilde N$ is attained by a unique geodesic segment  $\gamma$ from $x$ to  a point $y\in \partial \widetilde N$. Note that $\gamma$ is orthogonal to $\partial\widetilde N$. As $p$ is a local isometry, $p(\gamma)$ is a non-trivial line segment in $\R^n$ from  $p(x)$ to $p(y)$, which is orthogonal to $p(\partial \widetilde N)$ at $p(y)$. Since $\partial\widetilde N$ is convex, the vector in $\mathbb R^n$ from $p(y)$ to $p(x)$  is pointing inward (with respect to $p(\partial \widetilde N)$). Therefor $p(\gamma)$ lies entirely in the  inside\footnote{We have already shown that $p(\partial \widetilde N)$ is strictly convex smooth compact hypersurface in $\mathbb R^n$. It follows that $p(\partial \widetilde N)$ separates $\mathbb R^n$ into two parts. That is, $\mathbb R^n - p(\partial \widetilde N)$ consists of two connected components, exactly one of which is compact. We call the compact connected component of $\mathbb R^n - p(\partial \widetilde N)$ the inside of $p(\partial \widetilde N)$, and  the noncompact connected component of $\mathbb R^n - p(\partial \widetilde N)$ the outside of $p(\partial \widetilde N)$. } of  $p(\partial\widetilde N)$. Let $\alpha
	\colon [0, 1] \to p(\partial \widetilde N)$ be a smooth path in $p(\partial \widetilde N)$ with $\alpha(0)=p(y)$ and $\alpha(1)=p(x)$. Since $p$ is a covering map on $\partial\widetilde N$, $\alpha$ lifts uniquely to a path $\widetilde \alpha$ such that $\widetilde \alpha(0)=y$. As $p$ is a local isometry near $y$, there is a unique geodesic $\gamma_t$ connecting $y$ and  $\widetilde\alpha(t)$ for all sufficiently small $t\in [0, 1]$, which is mapped isometrically under the map $p$ to the line segment connecting $p(y)$ and $\alpha(t)$. Since $p$ is a local isometry everywhere, we can continue the construction of such geodesics $\gamma_t$ for all $t\in [0, 1]$ . In particular, $p(\gamma_1)$ coincides with $p(\gamma)$. By construction, $\gamma_1$ and $\gamma$ have the same length and point towards the same direction starting from $y$. It follows that   $x$  coincides with the other end point $\widetilde \alpha(1)$ of $\gamma_1$, which lies in $\partial \widetilde N$ by construction. This contradicts the assumption that $x$ lies in the interior of $\widetilde N$. This finishes the proof of the claim. Note that the same argument also proves that every point in the inside of $p(\partial\widetilde N)$ admits at least one preimage in $\widetilde N$.
	
	The interior $\widetilde N - \partial\widetilde N$  of $\widetilde N$ is connected and $p(\widetilde N - \partial\widetilde N)$ is disjoint from $p(\partial\widetilde N)$, so $p(\widetilde N - \partial\widetilde N)$ lies entirely in the inside of $p(\partial\widetilde N)$. To summarize, we see that $p(\widetilde N)$ is precisely the region enclosed by   the hypersurface $p(\partial\widetilde N)$ in $\mathbb R^n$. As $p(\partial\widetilde N)$ coincides with $\partial M$ up to an affine isometry, $p(\widetilde N)$ coincides with $M$ up to an affine isometry. Without loss of generality, we may assume that $p(\widetilde N)=M$. 
	
	Now we show that $p\colon \widetilde N \to p(\widetilde N) = M$ is a covering map. Indeed, if $z$ is a point in the interior of $M$, then its preimage $p^{-1}(z)$ consists of only interior points of $\widetilde N$. Let $\varepsilon$ be the distance from $z$ to $\partial M$. Then the $\varepsilon$-neighborhood of each point in $p^{-1}(z)$ is mapped  isometrically under the map $p$ to the $\varepsilon$-neighborhood of $z$. In particular,  the $(\varepsilon/2)$-neighborhoods of points in $p^{-1}(z)$ are disjoint in  $\widetilde N$. The same holds when $z$ lies in $\partial M$, as each point in its preimage lies in $\partial \widetilde N$. 
	
	As $M$ is simply connected,  $p$ has to be the trivial covering map, hence  an isometry. In particular, $p\colon \partial \widetilde N\to \partial M$ is a homeomorphism. Let $\pi\colon\widetilde N\to N$ be the corresponding covering map for the universal cover $\widetilde N$ of $N$. Note that $f\circ \pi=p$ on $\partial\widetilde N$. Therefore the restriction  $\pi$ on $\partial\widetilde N$ is injective. It follows that $\widetilde N=N$, $\pi$ is the identity map and $\partial f\coloneqq f|_{\partial N}$ is an isometry. 
	To summarize, we have proved that $N $ is a strictly convex domain in $\mathbb R^n$ such that,  up to an affine isometry, $N$ coincides with $M$ and $\partial f$ becomes the identity map. 
\end{proof}

\begin{remark}
	Since $\partial N$ is connected, $f\colon \partial N\to\partial M$ is a finite covering map for $n=2$, or an isometry if $n\geq 3$. As $N$ is flat and $f$ is distance-non-increasing and preserves second fundamental form, the mean curvature flow of $\partial N$ inside $N$ behaves the same as the mean curvature flow of $\partial M$ in $\R^n$. It follows from \cite[Theorem 1.1]{MR772132} that the mean curvature flow is smooth for all time until it converges to a point. This shows that $N$ is topologically a ball foliated by the hypersurfaces generated by the mean curvature flow of $\partial N$, hence simply connected. Now a similar argument from the above proof  shows that $N $ is a strictly convex domain in $\mathbb R^n$ such that,  up to an affine isometry, $N$ coincides with $M$ and $\partial f$ becomes the identity map.  This gives an alternative proof of Theorem \ref{thm:ballAlmostIso}.
\end{remark}

In general, although the map $p\colon N\to M$ we constructed in the proof of Theorem \ref{thm:ballAlmostIso} is an isometry, the original map $f$ may not be an isometry on the whole $N$. For example, if $N$ and $M$ are both the unit disk $\mathbb B^2$ in $\R^2$, then any smooth map $f\colon \mathbb B^2\to \mathbb B^2$ with   $\partial f = \id \colon \partial \mathbb B^2\to \partial \mathbb B^2$  satisfies the assumptions of  Theorem \ref{thm:ballAlmostIso}. Even if we further assume that $f$ is area-non-increasing on the whole $\mathbb B^2$, $f$ may still not be an isometry. For example, let  $f\colon \mathbb B^2 \to \mathbb B^2$ be given in polar coordinates by
$$re^{i\theta}\mapsto re^{i(\theta+\rho(r))}$$
where $\rho\colon [0,1]\to \mathbb R$ is any smooth function that vanishes near $0$ and $1$.
Note that such a map $f$ is always area-preserving and has degree one, but  certainly not an isometry in general. On the other hand, if we furthermore assume  $f$ to be distance non-increasing on the whole $\mathbb B^2$, then it turns out that $f$ itself has to be an isometry, which is the content of Theorem \ref{thm:balliso}.

Now let us prove Theorem \ref{thm:balliso}. 

\begin{proof}[{Proof of Theorem \ref{thm:balliso}}]
	With the same notation as in the proof of Theorem \ref{thm:ballAlmostIso}, we have an isometry $p\colon N\to M$ such that the map $h\coloneqq f\circ p^{-1}\colon M \to N \to M$ is distance non-increasing, and is equal to the identity map when restricted to $\partial M$. To prove the theorem, it suffices to show that any such map $h$ has to be the identity map on $M$. Let $x_1$ and $x_2$ be two arbitrary points on $\partial M$. Since $M$ is strictly convex, there is a unique   line segment $\ell$ connecting $x_1$ and $x_2$ that  lies entirely in $M$. Then $h(\ell)$ is a curve in $M$ connecting  $x_1$ and $x_2$, hence its length is at least the length of $\ell$. Since $h$ is distance non-increasing, it follows that $h$ maps   $\ell$ to itself isometrically. Note that all such line segments cover the whole $M$. This completes the proof.
	
\end{proof}

\end{document}